\documentclass[11pt]{article}
\usepackage[a4paper, left=2cm, right=2cm, top=2.54cm, bottom=2.54cm, heightrounded]{geometry}
\usepackage{amsmath, amsthm, amsfonts, amssymb, authblk, bm, graphicx, mathrsfs,bbm}
\usepackage[numbers]{natbib}
\usepackage[titletoc,title]{appendix}
\usepackage[hyphens]{url}
\usepackage[pdftex,colorlinks=true,hypertexnames=false]{hyperref}
\usepackage{tipa, color, enumerate}
\definecolor{darkblue}{rgb}{0,0,.6}
\hypersetup{colorlinks=true,citecolor=darkblue,linkcolor=darkblue,urlcolor=darkblue}

\allowdisplaybreaks[4]%允许align公式换页
\newtheoremstyle{assumption}
                {2pt}
                {2pt}
                {\itshape}
                {}
                {\bfseries}
                {}
                { }
                { } 
\numberwithin{equation}{section}
\newtheorem{theorem}{Theorem}[section]
\newtheorem{definition}[theorem]{Definition}
\newtheorem{lemma}[theorem]{Lemma}
\newtheorem{corollary}[theorem]{Corollary}
\newtheorem{proposition}[theorem]{Proposition}
\theoremstyle{remark}
\newtheorem{remark}[theorem]{Remark}
\theoremstyle{assumption}
\newtheorem{assumption}{}

\newenvironment{keywords}{{\bf Key words: }}{}
\newenvironment{AMS}{{\bf AMS subject classification: }}{}
\title{Classical solution to second-order Hamilton-Jacobi-Bellman equation and optimal feedback control for linear-convex problem}

\author{Jinghua Li and Zhiyong Yu \thanks{School of Mathematics, Shandong University, Jinan 250100, China. E-mails: {\color{darkblue}{\tt jinghua.li@mail.sdu.edu.cn}} (J. Li), {\color{darkblue}{\tt yuzhiyong@sdu.edu.cn}} (Z. Yu). The corresponding author is Zhiyong Yu.}  }
  
\date{\today}

\begin{document}
\maketitle
\begin{abstract}
In this paper, we are concerned with the classical solvability of a class of second-order Hamilton-Jacobi-Bellman equations (HJB equations) arising from stochastic optimal control problems with linear dynamics and uniformly convex cost functionals. 
By introducing the Hamiltonian system and extending the gradient descent method to a Hilbert space, we prove the existence and uniqueness of the optimal control under the uniform convexity condition. 
The regularity of the solution to the Hamiltonian system is obtained, including the derivatives with respect to the initial state and the Malliavin derivatives.
The connection between the Hamiltonian system and the value function is subsequently proven, enabling us to derive regularity properties of the value function via probabilistic techniques. 
Finally, by the dynamic programming principle, the value function is verified to be the unique classical solution to the HJB equation and the optimal feedback control is provided. 
These results generalize the classical linear-quadratic theory and provide a new insight into the regularity of the value function.
\end{abstract}


\vspace{0.5cm}


\section{Introduction}\label{sec1}

In this paper, we are concerned with the following second-order HJB equation, which is a fully nonlinear partial differential equation (PDE):
\begin{equation}\label{1:HJB}
\left\{
\begin{aligned}
& \frac {\partial V}{\partial t}(t,x)+(\mathcal {L}V)(t,x)+ H(t,x,\mathrm{D}_xV(t,x),\mathrm{D}^2_{xx}V(t,x))=0,\quad t\in [0,T],\\
& V(T,x)=g(x),
\end{aligned}
\right.
\end{equation}
where $V:[0,T]\times \mathbb R^n\to\mathbb R$ is the unknown function.
The operator $\mathcal L$ is defined as
$$
\begin{aligned}
(\mathcal {L}V)(t,x):=&\langle \mathrm{D}_xV(t,x),A(t)x+b(t)\rangle\\
&+\frac12\sum_{i=1}^d\langle \mathrm{D}^2_{xx}V(t,x)[C_i(t)x+\sigma_i(t)],C_i(t)x+\sigma_i(t)\rangle.
\end{aligned}
$$
The function $H(t,x,p,Q):=\inf_{u\in\mathbb R^m}\mathcal H(t,x,p,Q,u)$, where $\mathcal H:[0,T]\times\mathbb R^{n+n+n\times n+m}\to\mathbb R$ is defined as
$$
\begin{aligned}
\mathcal H(t,x,p,Q,u):= &\langle u,B(t)^\top p+\sum_{i=1}^dD_i(t)^\top Q[C_i(t)x+\sigma_i(t)]\rangle\\
&+\frac12\sum_{i=1}^d \langle D_i(t)^\top QD_i(t)u,u\rangle+l(t,x,u).
\end{aligned}
$$
Here, $A,C_1,...,C_d:[0,T]\to\mathbb R^{n\times n}$, $B,D_1,...,D_d:[0,T]\to\mathbb R^{n\times m}$, $b,\sigma_1,...,\sigma_d:[0,T]\to\mathbb R^n$, $g:\mathbb R^n\to\mathbb R$, $l:[0,T]\times \mathbb R^{n+m}\to\mathbb R$ are measurable functions
and the differential operators $\mathrm{D}_x,\mathrm{D}^2_{xx}$ are defined as follows:
\begin{equation}\label{1:do}
    \mathrm{D}_x:=\left(\frac {\partial }{\partial x_1},\frac {\partial }{\partial x_2},...,\frac {\partial }{\partial x_n}\right)^\top,~~\mathrm D^2_{xx}:=\mathrm D_x \mathrm D_x^\top=\left(\frac {\partial^2}{\partial x_i\partial x_j}\right)_{n\times n}.
  \end{equation} 
We denote $C=(C_1^\top,...,C_d^\top)^\top$, $D=(D_1^\top,...,D_d^\top)^\top$, $\sigma=(\sigma_1^\top,...,\sigma_d^\top)^\top$ for convenience. 

The above HJB equation \eqref{1:HJB} is related to the following stochastic linear-convex optimal control problem (LC problem).
The state process satisfies the following controlled linear stochastic differential equation (SDE):
\begin{equation}\label{1:SDE0}
\left\{
\begin{aligned}
\mathrm dX^{0,x,u}_s =& \big[A(s)X^{0,x,u}_s+B(s)u_s+b(s)\big] \, \mathrm ds \\
&+\sum_{i=1}^d \big[C_i(s)X^{0,x,u}_s+D_i(s)u_s+\sigma_i(s)\big]\, \mathrm dW^i_s,\quad s\in [0,T],\\
X^{0,x,u}_0 =& x\in\mathbb R^n,
\end{aligned}
\right.
\end{equation} 
where $W:=(W^1,...,W^d)^{\top}$ is a $d$-dimensional standard Brownian motion defined on a completed filtered probability space $(\Omega, \mathcal F, \mathbb F, \mathbb P)$ with $\mathbb F=\{\mathcal F_t\}_{0\leq t\leq T}$ being the augmented natural filtration generated by $W$.
The admissible control $u$ belongs to $\mathscr U_{\mathrm {ad}}[0,T]$ where
$$
\mathscr U_{\mathrm {ad}}[0,T]:=\bigg\{\phi:\Omega\times[0,T]\to\mathbb R^m\bigg\vert\phi\mbox{ is }\mathbb F\mbox{-adapted such that }\mathbb E\int_0^T\vert\phi_s\vert^2\, \mathrm ds<\infty\bigg\}.
$$
The cost functional is defined as
\begin{equation} \label{1:cf}
J(0,x;u)=\mathbb E\left[g(X^{0,x,u}_T)+\int_0^Tl(s,X^{0,x,u}_s,u_s)\, \mathrm ds\right],
\end{equation} 
which is assumed to satisfy the following uniform convexity condition:
\begin{equation} \label{1:uc}
\mbox{The mapping }u\mapsto J(0,x;u)\mbox{ is uniformly convex.} 
\end{equation} 
The aim of the LC problem is to find the optimal control $\bar u^{0,x}$ such that 
$$
J(0,x;\bar u^{0,x})=\inf_{u\in\mathscr U_{\mathrm{ad}}[0,T]}J(0,x;u).
$$

When $g,l$ take quadratic forms, i.e., $g(x)=\langle Gx,x\rangle$ and $l(t,x,u)=\langle Q(t)x,x\rangle+2\langle S(t)x,u\rangle+\langle R(t)u,u\rangle$, where $G,Q,S,R$ are matrix valued with proper dimensions,
the above optimal control problem is reduced to a linear-quadratic optimal control problem (LQ problem), which has attracted attentions of many scholars. 
Due to the distinctive structure of the LQ problem, the value function $V$ is verified to admit a quadratic form, i.e., $V(t,x)=\langle P(t)x,x\rangle$, where $P$ is symmetric matrix valued. 
Accordingly, the HJB equation \eqref{1:HJB} is reduced to a Riccati equation, which characterizes $P$ and is nonlinear in $P$. 
The related research is fruitful and well-developed. 
For example, one could refer to Sun and Yong \cite{sy18,sy19} for the case of deterministic coefficients. 
Tang \cite{t03,t15} first established the well-posedness results of the stochastic Riccati equations (SREs), which arise from LQ problems with random coefficients. 
Sun et al. \cite{sxy21} proved the well-posedness of SREs under the uniform convexity condition.
One could further refer to \cite{gm09,ho08,hly15,sxy21,wyy19,y13} for the corresponding results under various settings.

In the present paper, we focus on the LC problems, which are more general than the LQ cases.
For the open-loop optimal control of the LC problem, the case where all coefficients are stochastic is considered, which is beneficial to subsequent estimates. 
By introducing the adjoint equation, which is a linear backward stochastic differential equation (BSDE), the cost functional is proven to be Frech\'et differentiable. 
Since $\mathscr U_{\mathrm{ad}}[0,T]$ is a Hilbert space, the Frech\'et derivative is represented by an operator $\mathcal D$, which maps $\mathscr U_{\mathrm {ad}}[0,T]$ into itself. 
Combining the controlled SDE, the adjoint BSDE and the stationary condition $\mathcal D[u]=0$, a non-linear Hamiltonian system is introduced to characterize the optimal control $\bar u^{0,x}$ in the open-loop form, whose solution consists of four processes $(\bar X^{0,x},\bar Y^{0,x},\bar Z^{0,x},\bar u^{0,x})$.
Therefore, to find the open-loop optimal control of the LC problem is converted into solving the Hamiltonian system.

Historically, when assuming $l$ is uniformly convex in $u$, the optimal control $\bar u^{0,x}$ can be represented by $(\bar X^{0,x},\bar Y^{0,x},\bar Z^{0,x})$ based on the equation $\mathcal D[u]=0$.
By substituting the representation into the controlled SDE and the adjoint BSDE, the Hamiltonian system is rewritten as a nonlinear forward-backward stochastic differential equation (FBSDE).
For the well-posedness of such an FBSDE, a widely used method is the so-called {\em method of continuation}, which was established by Hu and Peng \cite{hp95}, Yong \cite{y97} and Peng and Wu \cite{pw99}.
The main idea is that, under suitable monotonicity conditions, the well-posedness of FBSDEs still holds when the coefficients are slightly perturbed.
Therefore, by continuously perturbing the coefficients from an initial case with known well-posedness, the well-posedness of the desired FBSDE is obtained.
When considering the well-posedness of the Hamiltonian systems arising from LC problems, Liu et al. \cite{lnwy25} introduced a kind of nonlinear domination-monotonicity conditions, which coincide with the Hamiltonian system.
For the mean field case, Carmona and Delarue \cite{cd15} also obtained the unique solution to the Hamiltonian system by introducing the convex functions of measures.
Bensoussan et al. \cite{bhty24} considered the case where all coefficients are defined on Hilbert spaces and provided the well-posedness of FBSDEs under the $\beta$-monotonicity conditions, also obtaining the optimal control for the mean field case.

In this paper, only the uniform convexity condition \eqref{1:uc} is taken into account, under which $l$ may be not uniformly convex in $u$.
Therefore, the Hamiltonian system may not be converted to an FBSDE.
Specifically, the uniform convexity condition \eqref{1:uc} contains the following two special cases (compared with the LQ case, the positive definiteness of matrices is replaced by the convexity of functions):
\begin{itemize}
  \item {\em Case 1}. There is a constant $\delta>0$ such that, for almost all $(t,\omega)\in[0,T]\times\Omega$, the functions $g(x)$ and $l(t,x,u)-\frac{\delta}{2}\vert u\vert^2$ are convex in $(x,u)$.
  \item {\em Case 2}. There is a constant $\delta>0$ such that, for almost all $(t,\omega)\in[0,T]\times\Omega$, the functions $g(x)-\frac{\delta}{2}\vert x\vert^2$ and $l(t,x,u)$ are convex in $(x,u)$ and $D(t)^{\top}D(t)\geq\delta I_m$.
\end{itemize}
Here, $D(t)^{\top}D(t)\geq\delta I_m$ means that $D(t)^{\top}D(t)-\delta I_m$ is positive semidefinite, where $I_m$ is the $m$-order identity matrix.
The monotonicity conditions mentioned above can be verified under Case 1.
However, when $l$ is not uniformly convex in $u$ (for example, Case 2), the representation of the optimal control $\bar u^{0,x}$ by $(\bar X^{0,x},\bar Y^{0,x},\bar Z^{0,x})$ is unavailable, leading to the inapplicability of the method of continuation.

\label{case}
To overcome the difficulty, we extend the gradient descent method to the Hilbert space $\mathscr U_{\mathrm {ad}}[0,T]$ to prove the well-posedness of the Hamiltonian system.
In detail, we introduce a mapping $\mathcal T$ from $\mathscr U_{\mathrm {ad}}[0,T]$ into itself as follows:
\begin{equation} 
\mathcal T[u]:=u-\eta\mathcal D[u]
\end{equation} 
with an undetermined parameter $\eta>0$ which is called the learning rate, 
i.e., for any given initial control $u$, we update it by subtracting a scaled Frech\'et derivative of the cost functional.  
Due to the uniform convexity of the cost functional \eqref{1:cf}, we can select a suitable learning rate $\eta$ such that $\mathcal T$ is contractive.
As a result, the unique fixed point of $\mathcal T$ is verified to be the unique optimal control of the LC problem. 
Compared with the method of continuation, which needs multi-step fixed point iteration, our method is more direct and efficient to find the optimal control only by one-step iteration.  

For the optimal feedback control, when considering deterministic control systems, You \cite{yy87,yy97} introduced a kind of quasi-Riccati equations to characterize the optimal feedback control for the LC problem with $l$ being the following form:
\[
l(t,x,u)=M(t,x)+ \langle R(t)u,u\rangle,  
\]
where $M$ is a function that is convex in $x$, and $R$ is a mapping whose values are symmetric matrices.
Under our assumption, the optimal feedback control is related to the HJB equation \eqref{1:HJB}.
We further consider a family of parameterized optimal control problems:
For any given initial pair $(t,x)\in[0,T]\times \mathbb R^n$, minimize the following cost functional:
\begin{equation} 
J(t,x;u)=\mathbb E\left[g(X^{t,x,u}_T)+\int_t^Tl(s,X^{t,x,u}_s,u_s)\, \mathrm ds\right],
\end{equation} 
subject to the following controlled linear SDE:
\begin{equation}\label{1:SDEt}
\left\{
\begin{aligned} 
\mathrm dX^{t,x,u}_s=& \big[A(s)X^{t,x,u}_s+B(s)u_s+b(s)\big] \, \mathrm ds\\
&+\sum_{i=1}^d \big[C_i(s)X^{t,x,u}_s+D_i(s)u_s+\sigma_i(s)\big]\, \mathrm dW^i_s,\quad s\in [t,T],\\
X^{t,x,u}_t =& x,
\end{aligned}
\right.
\end{equation}
over the admissible control set $\mathscr U_{\mathrm{ad}}[t,T]$, where
$$
\mathscr U_{\mathrm{ad}}[t,T]:=\bigg\{\phi:\Omega\times[t,T]\to\mathbb R^m\bigg\vert\phi\mbox{ is }\mathbb F\mbox{-adapted such that }\mathbb E\int_t^T\vert\phi_s\vert^2\, \mathrm ds<\infty\bigg\}.
$$
The optimal control problem with the parameter $(t,x)$ is formulated as follows:

\noindent {\bf Problem (LC)$_{(t,x)}$}. Find the optimal control $\bar u^{t,x}\in\mathscr U_{\mathrm {ad}}[t,T]$, such that
$$
J(t,x;\bar u^{t,x})=\inf_{u\in\mathscr U_{\mathrm {ad}}[t,T]}J(t,x;u).
$$

Applying the previous well-posedness result of the Hamiltonian system, the optimal control $\bar u^{t,x}$ is immediately obtained.
We define the so-called value function $V:[0,T]\times \mathbb R^n\to\mathbb R$ by
\begin{equation} 
V(t,x):=\inf_{u\in\mathscr U_{\mathrm {ad}}[t,T]}J(t,x;u)=J(t,x;\bar u^{t,x}).
\end{equation} 
By the classical control theory, the value function $V$ is the viscosity solution of the HJB equation \eqref{1:HJB} under some appropriate conditions (see Yong and Zhou \cite[Chapter 4, Theorem 5.2]{yz99} for example). 
For the viscosity solution of HJB equations, please refer to Peng \cite{p92} for the case where the cost functional is characterized by a BSDE, Ekren et al. \cite{ektz14} for the path dependent case and Burzoni et al. \cite{birs20} for the mean field case.
In the present paper, we aim to prove the unique solvability of the HJB equation \eqref{1:HJB} in the classical sense, which requires us to obtain more regularity of $V$ than the viscosity sense.

Based on the well-posedness result of the Hamiltonian system, we obtain the derivatives with respect to $x$ and the Malliavin derivatives of the solution $(\bar X^{t,x},\bar Y^{t,x},\bar Z^{t,x},\bar u^{t,x})$.
The derivatives with respect to $x$, denoted by $(\nabla X^{t,x},\nabla Y^{t,x},\nabla Z^{t,x},\nabla u^{t,x})$, are verified to satisfy a linear Hamiltonian system. 
With the help of the differentiability of $(\bar X^{t,x},\bar u^{t,x})$, the value function $V$ is proven to be once differentiable with respect to $x$. 
Then, by applying It\^o's formula to $(\nabla X^{t,x}_s)^{\top} \bar Y^{t,x}_s$ on $[t,T]$ and comparing to $\mathrm D_xV(t,x)$, the key connection between the Hamiltonian system and the value function is identified as follows: 
\begin{equation}\label{1:cnt}
  \mathrm D_xV(t,x)=\bar Y^{t,x}_t.
\end{equation}
In the general case where the controlled SDE \eqref{1:SDE0} is nonlinear, Yong and Zhou \cite[Chapter 5, Theorem 4.1]{yz99} assumed that the value function is regular enough ($\mathrm D_xV\in C^{1,2}([0,T]\times\mathbb R^n;\mathbb R^n)$) and obtained the same connection by applying It\^o's formula to $\mathrm D_xV(s,\bar X^{t,x}_s)$. 
However, it is unnecessary to impose such stringent regularity conditions under our setting. 
On the contrary, we aim to obtain the regularity of $V$ by the connection \eqref{1:cnt} via probabilistic techniques. 

\label{open}
Here, we would like to emphasize that the connection \eqref{1:cnt} is crucial for obtaining the regularity of the value function $V$. 
In the literature, when using the probabilistic method to study the well-posedness of second-order PDEs, for example \cite{pp92,p92,wy14,wxy25}, the connection $\psi(t,x)=\bar Y^{t,x}_t$ between the solution $\psi$ to the PDE and the solution $\bar Y^{t,x}$ to the FBSDE is generally considered. 
Therefore, to obtain the twice differentiability of $\psi$, the second-order derivatives of $(\bar X^{t,x},\bar Y^{t,x},\bar Z^{t,x})$ should be taken into consideration (see Wu et al. \cite{wxy25} for example).
In this case, the $L^p$ ($p\geq 4$) estimates for FBSDEs are required, which are still open for coupled FBSDEs under general conditions.
However, due to the connection \eqref{1:cnt}, we only need the first-order derivatives of $(\bar X^{t,x},\bar Y^{t,x},\bar Z^{t,x})$ to obtain the second-order derivatives of the value function $V$, which is exempt from $L^p$ estimates.

By the differentiability of $\bar Y^{t,x}_t$ with respect to $x$, we further derive that $V$ is twice differentiable with respect to $x$ and $\mathrm D^2_{xx}V(t,x)=\nabla Y^{t,x}_t$. 
To obtain the continuity of $\mathrm D^2_{xx}V$, the continuity of $\nabla Y^{t,x}_t$ with respect to $(t,x)$ is needed.
Due to the absence of $L^p$ ($p\geq4$) estimates for Hamiltonian systems, Kolmogorov's continuity theorem is not available to obtain the version of $\nabla Y^{t,x}_s$ jointly continuous in $(t,s,x)$.
Inspired by Wu et al. \cite{wxy25}, we prove the continuous dependence of $\nabla Y^{t,x}$ with respect to $(t,x)$ only by $L^2$ estimate, obtaining the continuity of $\mathrm D^2_{xx}V$.

Moreover, with the help of the results in Sun et al. \cite{sxy21}, the following regular property of $V$ is proven:
There exists a constant $\delta>0$ such that, for all $(t,x,u)\in[0,T]\times\mathbb R^m$,
\begin{equation}\label{1:pov}
  \mathrm D^2_{uu}l(t,x,u)+\sum_{i=1}^d D_i(t)^\top \mathrm D^2_{xx}V(t,x)D_i(t)\geq\delta I_m.
\end{equation}
Here, similarly to \eqref{1:do}, the differential operators are defined as
\begin{equation}\label{1:do2}
\mathrm D_u:=\left(\frac {\partial }{\partial u_1},\frac {\partial }{\partial u_2},...,\frac {\partial }{\partial u_m}\right)^\top\mbox{ and }\mathrm D^2_{uu}:=\mathrm D_u\mathrm D_u^\top.
\end{equation}
By \eqref{1:pov} and the implicit function theorem, there exists a function $\bar{\mathbbm u}(t,x)$ such that $\bar{\mathbbm u}(t,x)$ is the minimizer of $\mathcal H(t,x,\mathrm{D}_xV(t,x),\mathrm{D}^2_{xx}V(t,x),u)$ for any $(t,x)\in[0,T]\times\mathbb R^n$,
Combining the Malliavin derivatives, $(\bar Z^{t,x},\bar u^{t,x})$ is proven to admit a continuous version and the optimal control $\bar u^{t,x}$ admits the following state feedback form: 
$$\bar u^{t,x}_{s}=\bar{\mathbbm u}(s,\bar X^{t,x}_{s}).$$

Finally, we prove the dynamic programming principle. 
With the help of the dynamic programming principle, the value function $V$ is verified to be differentiable with respect to $t$ and satisfies the HJB equation \eqref{1:HJB} in the classical sense, leading to the existence of the HJB equation \eqref{1:HJB}.
By introducing the weak formulation, we provide the optimal feedback control and introduce the corresponding closed-loop system. 
The stochastic verification theorem in the weak formulation is proven, which ensures the uniqueness of the HJB equation \eqref{1:HJB}.
Based on the existence and uniqueness of the HJB equation \eqref{1:HJB}, we prove that the closed-loop system admits a unique strong solution, which demonstrates that the optimal feedback control in the weak formulation is also optimal in the strong formulation.

Now, we would like to highlight the innovations and features of this paper as follows.
\begin{itemize}
  \item The unique classical solution to the HJB equation \eqref{1:HJB} is obtained under the uniform convexity condition. 
  Moreover, we provide the optimal feedback control for Problem (LC)$_{(t,x)}$ in both strong and weak formulations.  
  \item The nonlinear Hamiltonian system for Problem (LC)$_{(t,x)}$ is introduced, which is a generalization of the linear Hamiltonian systems arising from LQ problems. 
  By extending the gradient descent method to the Hilbert space, the well-posedness of the Hamiltonian system is proven under the uniform convexity condition, which is weaker than the monotonicity conditions used in the method of continuation.
  \item The connection \eqref{1:cnt} between the Hamiltonian system and the value function is identified, which ensures us to obtain the twice differentiability with respect to $x$ of the value function via the probabilistic techniques. 
  Moreover, the connection \eqref{1:cnt} provides a new insight into the relationship between FBSDEs and second-order PDEs.
\end{itemize}

The rest of this paper is organized as follows. 
In Section \ref{sec2}, we prove the well-posedness of the Hamiltonian system under the uniform convexity condition. 
In Section \ref{sec3}, the regularity of the solution to the Hamiltonian system is presented, including the derivatives with respect to $x$ and the Malliavin derivatives. 
In Section \ref{sec4}, based on the connection \eqref{1:cnt}, some properties of value function are proven.
In Section \ref{sec5}, we prove the unique classical solvability for the HJB equation \eqref{1:HJB} and provide the optimal feedback control.
The proofs of two lemmas are placed in Appendices \ref{APP-A} and \ref{APP-B}, respectively.
We provide a sufficient condition for the convexity of the value function $V$ in Appendix \ref{APP-C}.

\section{Hamiltonian system and open-loop optimal control}\label{sec2}
First of all, we would like to introduce some notations and spaces used in this paper.
Let $\mathbb R^n$ be the $n$-dimensional Euclidean space equipped with the Euclidean inner product $\langle \cdot,\cdot\rangle$ and the induced norm $\vert\cdot\vert$ . 
Specially, let $\mathbb R^{n\times m}$ be the set of all $(n\times m)$ matrices and $\mathbb S^n$ be the set of all $(n\times n)$ symmetric matrices.

In addition, we further introduce the following spaces defined on the filtered probability space $(\Omega, \mathcal F, \mathbb F, \mathbb P)$.
\begin{align*}
L^2_{\mathcal F_t}(\Omega;\mathbb R^n):=&\big\{\xi:\Omega\to\mathbb R^n\big\vert\xi\mbox{ is }\mathcal F_t\mbox{-measurable such that }\mathbb E[\vert\xi\vert^2]<\infty\big\}.\\
L^{\infty}_{\mathcal F_t}(\Omega;\mathbb R^n):=&\big\{\xi:\Omega\to\mathbb R^n\big\vert\xi\mbox{ is }\mathcal F_t\mbox{-measurable and essentially bounded}\big\}.\\
L^2_{\mathbb F}(t,T;\mathbb R^n):=&\bigg\{\phi:[t,T]\times\Omega\to\mathbb R^n\bigg\vert\phi\mbox{ is }\mathbb F\mbox{-adapted such that }\mathbb E\int_t^T\vert\phi_s\vert^2\, \mathrm ds<\infty\bigg\}.\\
L^{1,2}_{\mathbb F}(t,T;\mathbb R^n):=&\bigg\{\phi:[t,T]\times\Omega\to \mathbb R^n\bigg\vert\phi\mbox{ is }\mathbb F\mbox{-adapted such that }\\
&\mathbb E\left[\left(\int^T_t\vert \phi_s\vert\, \mathrm d t\right)^2\right]<\infty\bigg\}.\\
L^\infty_{\mathbb F}(t,T;\mathbb R^n):=&\big\{\phi:[t,T]\times\Omega\to \mathbb R^n\bigg\vert\phi\mbox{ is }\mathbb F\mbox{-adapted and essentially bounded}\big\}.\\
L^2_{\mathbb F}(\Omega;C(t,T;\mathbb R^n)):=&\bigg\{\phi:[t,T]\times\Omega\to\mathbb R^n\bigg\vert\phi\mbox{ is }\mathbb F\mbox{-adapted and pathwise continuous}\\
&\mbox{a.s. such that }\mathbb E\left[\sup_{s\in[t,T]}\vert\phi_s\vert^2\right]<\infty\bigg\}.
\end{align*}
Note that the admissible control set $\mathscr U_{\mathrm{ad}}[t,T]=L^2_{\mathbb F}(t,T;\mathbb R^m)$, which is defined in Section \ref{sec1}.
Moreover, we introduce the following product space as the solution space of the Hamiltonian system:
$$
\mathcal M[t,T]:= L^2_{\mathbb F}(\Omega;C(t,T;\mathbb R^n))\times L^2_{\mathbb F}(\Omega;C(t,T;\mathbb R^n))\times L^2_{\mathbb F}(t,T;\mathbb R^{nd})\times \mathscr U_{\mathrm {ad}}[t,T].
$$

In this section, the initial pair $(0,x)$ is fixed, and we would like to omit it for simplification. 
Moreover, we consider the case where all coefficients are stochastic, which is more general and will be beneficial later. 
Specifically, the following assumptions are introduced for Problem (LC)$_{(0,x)}$ with random coefficients. 
\begin{assumption}
  \label{A1}
The matrix valued stochastic processes $A, B, C, D$ are $\mathbb F$-adapted and essentially bounded. 
For any $(x,u)\in\mathbb R^{n+m}$, $g(x)$ is $\mathcal F_T$-measurable and $l(\cdot,x,u)$ is $\mathbb F$-adapted. 
For almost all $(t,\omega)\in[0,T]\times\Omega$, $g,l$ are continuously differentiable with respect to $(x,u)$. 
Moreover, all derivatives are uniformly Lipschitz continuous with respect to $(x,u)$, i.e., there exists a constant $K>0$ such that 
\begin{equation*}
\begin{aligned}
&\vert\mathrm D_xg(x_1)-\mathrm D_xg(x_2) \vert\leq K\vert x_1-x_2\vert,\\
&\vert\mathrm D_xl(t,x_1,u_1)-\mathrm D_xl(t,x_2,u_2)\vert+\vert\mathrm D_ul(t,x_1,u_1)-\mathrm D_ul(t,x_2,u_2)\vert\\
\leq& K(\vert x_1-x_2\vert+\vert u_1-u_2\vert),
\end{aligned}
\end{equation*}
for almost all $(t,\omega)\in[0,T]\times\Omega$ and any $(x_1,u_1),(x_2,u_2)\in\mathbb R^{n+m}$.
\end{assumption}
\begin{assumption}
  \label{A2}
The stochastic processes $b\in L^{1,2}_{\mathbb F}(0,T;\mathbb R^n)$, $\sigma\in L^2_{\mathbb F}(0,T;\mathbb R^{nd})$, $\mathrm D_xl(\cdot,0,0)\in L^{1,2}_{\mathbb F}(0,T;\mathbb R^n)$, $\mathrm D_ul(\cdot,0,0)\in L^2_{\mathbb F}(0,T;\mathbb R^m)$, 
the random variable $\mathrm D_xg(0)\in L^2_{\mathcal F_T}(\Omega,\mathbb R^n)$ and $\mathbb E[\vert g(0)\vert+\int_0^T \vert l(t,0,0)\vert \, \mathrm dt]$ exists and is finite.
\end{assumption}

Under Assumptions \ref{A1} and \ref{A2}, the controlled SDE \eqref{1:SDE0} admits a unique strong solution $X^{u}\in L^2_{\mathbb F}(\Omega;C(0,T;\mathbb R^n))$ for any admissible control $u\in\mathscr U_{\mathrm{ad}}[0,T]$. 
Therefore, the cost functional $J$ is well-defined by the following estimate:
\begin{equation}
\begin{aligned}
&\vert l(t,x,u)\vert\leq K\vert x\vert^2+K\vert u\vert^2+\vert\mathrm D_xl(t,0,0)\vert\vert x\vert+\vert\mathrm D_ul(t,0,0)\vert\vert u\vert+\vert l(t,0,0)\vert,\\
&\vert g(x)\vert\leq K\vert x\vert^2 +\vert\mathrm D_xg(0)\vert\vert x\vert+\vert g(0)\vert,
\end{aligned}
\end{equation}
where $K$ is the Lipschitz constant of $\mathrm D_xg,\mathrm D_xl,$ and $\mathrm D_ul$.

To find the optimal control, we define the Hamiltonian $\mathbb H:[0,T]\times\Omega\times\mathbb R^{n+n+nd+m}\to\mathbb R$ as follows:
\begin{equation}
\mathbb H(t,x,y,z,u)=\langle A(t)x+B(t)u+b(t),y\rangle+\sum_{i=1}^d\langle C_i(t)x+D_i(t)u+\sigma_i(t),z_i\rangle+l(t,x,u),
\end{equation}
where $z:=(z_1^\top,z_2^\top,...,z_d^\top)^\top$.
For any given $u\in\mathscr U_{\mathrm {ad}}[0,T]$, we introduce the following adjoint equation:
\begin{equation}
\label{2:BSDE}
\left\{
\begin{aligned}
\mathrm d Y^{u}_t=&-\mathrm D_x\mathbb H (t,X^{u}_t,Y^{u}_t,Z^{u}_t,u_t)\, \mathrm dt+\sum_{i=1}^{d}Z^{u,i}_t\, \mathrm dW^i_t,\quad t\in [0,T],\\
Y^{u}_T=&\mathrm D_x g(X^{u}_T),
\end{aligned}
\right.
\end{equation}
where we denote $Z^{u}:=((Z^{u,1})^\top,...,(Z^{u,d})^\top)^\top$.
Noting that 
$$
\mathrm D_x\mathbb H(t,x,y,z,u)=A(t)^\top y+C(t)^\top z+\mathrm D_xl(t,x,u),
$$
the above equation is a linear BSDE.
Under Assumptions \ref{A1} and \ref{A2}, $A^\top,C^\top$ are bounded,
$$
\mathbb E\left[\left(\int_0^T\vert \mathrm D_xl(t,X^u_t,u_t)\vert\, \mathrm dt\right)^2\right]<\infty\mbox{ and }\mathbb E\left[\vert \mathrm D_xg(X^u_T)\vert^2\right]<\infty.
$$
By the classical BSDE theory, the BSDE  \eqref{2:BSDE} admits a unique solution $(Y^{u},Z^{u})\in L^2_{\mathbb F}(\Omega;C(0,T;\mathbb R^n))\\\times L^2_{\mathbb F}(0,T;\mathbb R^{nd})$ for any admissible control $u\in\mathscr U_{\mathrm {ad}}[0,T]$. 

The following proposition shows that $J$ is Frech\'et differentiable. 
Since $\mathscr U_{\mathrm{ad}}[0,T]$ is a Hilbert space, its Frech\'et derivative should be a mapping from $\mathscr U_{\mathrm{ad}}[0,T]$ into itself by Riesz's representation theorem. 
Specifically, we are able to represent the Frech\'et derivative using $(Y^{u},Z^{u})$.
\begin{proposition}
\label{prop:fd}
Under Assumptions \ref{A1} and \ref{A2}, the cost functional $J$ is Frech\'et differentiable with respect to $u$, and its Frech\'et derivative at $u$ is represented as
\begin{equation}\label{prop:fd:1}
\mathcal D[u]_t=\mathrm D_u\mathbb H(t,X^{u}_t,Y^{u}_t,Z^{u}_t,u_t)=B(t)^\top Y^{u}_t+D(t)^\top Z^{u}_t+\mathrm D_ul(t,X^{u}_t,u_t),\quad t\in[0,T].
\end{equation}
Moreover, for any $v\in\mathscr{U}_{\mathrm{ad}}[0,T]$, the following equality holds:
\begin{equation}\label{prop:fd:2}
\begin{aligned}
\mathbb E\int_0^T\langle \mathcal D[u]_t,v_t-u_t\rangle\, \mathrm d t=&\mathbb E\Bigg[\langle \mathrm D_xg(X^u_T),X^v_T-X^u_T\rangle\\
&+\int_0^T\left\langle \begin{pmatrix}
\mathrm D_xl(t,X^u_t,u_t)\\
\mathrm D_ul(t,X^u_t,u_t)
\end{pmatrix},\begin{pmatrix}
X^{v}_t-X^u_t\\
v_t-u_t
\end{pmatrix}\right\rangle\, \mathrm d t\Bigg].
\end{aligned}
\end{equation}
\end{proposition}
\begin{proof}
For any given $u,v\in\mathscr U_{\mathrm {ad}}[0,T]$, we have
\begin{equation}\label{pf:prop:fd:1}
\begin{aligned}
J(v)-J(u)=&\mathbb E\Bigg[g(X^{v}_T)-g(X^{u }_T)+ \int_0^T[l(t,X^v_t,v_t)-l(t,X^u_t,u_t)]\, \mathrm d t\Bigg]\\
=&\mathbb E\Bigg[\langle \mathrm D_xg(X^u_T),X^v_T-X^u_T\rangle\\
&+\int_0^T\left\langle \begin{pmatrix}
\mathrm D_xl(t,X^u_t,u_t)\\
\mathrm D_ul(t,X^u_t,u_t)
\end{pmatrix},\begin{pmatrix}
X^{v}_t-X^u_t\\
v_t-u_t
\end{pmatrix}\right\rangle\, \mathrm d t \Bigg]+R_{v-u}.
\end{aligned}
\end{equation}
By the linearity of the SDE \eqref{1:SDE0}, it is clear that $X^{u^\gamma}=X^u+\gamma(X^v-X^u)$ where we denote $u^\gamma:=u+\gamma(v-u)$ for any $\gamma\in[0,1]$.
Therefore, the remainder term $R_{v-u}$ can be represented as
\begin{equation}
\begin{aligned}
R_{v-u}=&\mathbb E\Bigg[\int_0^1\langle \mathrm D_xg(X^{u^\gamma}_T)-\mathrm D_xg(X^u_T),X^v_T-X^u_T\rangle\, \mathrm d\gamma\\
&+\int_0^T\int_0^1\left\langle 
\begin{pmatrix}
\mathrm D_xl(t,X^{u^\gamma}_t,u^\gamma_t)-\mathrm D_xl(t,X^u_t,u_t)\\
\mathrm D_ul(t,X^{u^\gamma}_t,u^\gamma_t)-\mathrm D_ul(t,X^u_t,u_t)
\end{pmatrix},\begin{pmatrix}
X^{v}_t-X^u_t\\
v_t-u_t
\end{pmatrix}\right\rangle\, \mathrm d\gamma\, \mathrm d t \Bigg].
\end{aligned}
\end{equation}
Using the standard estimate for SDEs, there exists a constant $K>0$ such that
\begin{equation}\label{pf:prop:fd:2}
\mathbb E\Big[\sup_{0\leq t\leq T}\vert X^v_t-X^u_t\vert^2\Big]\leq K\mathbb E\int_0^T\vert  v_t-u_t\vert^2\, \mathrm d t.
\end{equation}
Therefore, by the Lipschitz continuity of $\mathrm D_xg,\mathrm D_xl,\mathrm D_ul$, we obtain
\begin{equation}
\begin{aligned}
\vert R_{v-u}\vert\leq&K\mathbb E\Bigg[\int_0^1\gamma \vert X^v_T-X^u_T\vert^2\, \mathrm d\gamma+\int_0^T\int_0^1\gamma\left[\vert X^v_t-X^u_t\vert +\vert v_t-u_t\vert \right]^2\, \mathrm d\gamma\, \mathrm d t \Bigg]\\
\leq& K\mathbb E\Bigg[\sup_{0\leq t\leq T}\vert X^v_t-X^u_t\vert^2+\int_0^T\vert v_t-u_t\vert^2\, \mathrm d t \Bigg]\leq K\Vert  v-u\Vert_{L^2}^2,
\end{aligned}
\end{equation}
where $\Vert v-u\Vert_{L^2}:=\left(\mathbb E\int_0^T\vert  v_s-u_s\vert^2\, \mathrm d s\right)^{\frac12}$ is the $L^2$ norm of $(v-u)$.

Applying It\^o's formula to $\langle X^{v}_t-X^u_t,Y^u_t\rangle$ on $[0,T]$ leads to 
\begin{equation}\label{pf:prop:fd:3}
\begin{aligned}
 &\mathbb E\Bigg[\langle \mathrm D_xg(X^u_T),X^v_T-X^u_T\rangle+\int_0^T\langle \mathrm D_xl(t,X^u_t,u_t),X^{v}_t-X^u_t\rangle\, \mathrm d t\Bigg]\\
 =&\mathbb E \int_0^T\langle v_t-u_t,B(t)^\top Y^u_t+D(t)^\top Z^u_t\rangle \, \mathrm d t.
\end{aligned}
\end{equation}
Comparing to \eqref{pf:prop:fd:1}, we deduce 
\begin{equation}
J(v)-J(u)=\mathbb E\int_0^T\langle \mathrm D_u\mathbb H(t,X^{u}_t,Y^{u}_t,Z^{u}_t,u_t),v_t-u_t\rangle\, \mathrm d t+R_{v-u},
\end{equation}
which leads to
$$
\begin{aligned}
&\lim_{\Vert v-u\Vert_{L^2}\to 0}~~\frac {\left\vert J(v)-J(u)-\mathbb E\int_0^T\langle\mathrm D_u\mathbb H(t,X^{u}_t,Y^{u}_t,Z^{u}_t,u_t),v_t-u_t \rangle\, \mathrm d t \right\vert}{\Vert v-u\Vert_{L^2}}\\
=&\lim_{\Vert v-u\Vert_{L^2}\to 0}~~\frac {\vert R_{v-u}\vert}{\Vert v-u\Vert_{L^2}}=0.
\end{aligned}
$$
Consequently, $J$ is Frech\'et differentiable and its Frech\'et derivative is given by \eqref{prop:fd:1}. 
The equality \eqref{prop:fd:2} comes from \eqref{pf:prop:fd:3}.
The proof is complete.
\end{proof}

For the given initial pair $(0,x)$, the cost functional $J$ is called convex, if the mapping $u\mapsto J(u)$ is convex.
The cost functional $J$ is termed uniformly convex, if there exists a constant $\delta >0$ such that $J(u)-\frac\delta 2\Vert u\Vert^2_{L^2}$ is convex.
For the sufficient condition of uniform convexity, we have introduced Case 1 and Case 2 in Page \pageref{case}.
The proof is similar to the LQ case (see Sun et al. \cite[Proposition 7.1]{sxy21} for example).

The following lemma provides a characterization for the convexity of cost functional via its Frech\'et derivative $\mathcal D$.
\begin{lemma}
\label{lem:1oc}
Let Assumptions \ref{A1} and \ref{A2} hold. Then,

(i) The cost functional $J$ is convex if and only if, for any $u,v\in\mathscr U_{\mathrm{ad}}[0,T]$,
\begin{equation}
\label{lem:1oc:1}
\mathbb E\int_0^T\langle \mathcal D[u]_t-\mathcal D[v]_t ,u_t-v_t\rangle\, \mathrm d t\geq 0.
\end{equation}

(ii) The cost functional $J$ is uniformly convex if and only if there exists a constant $\delta>0$ such that, for any $u,v\in\mathscr U_{\mathrm{ad}}[0,T]$,
\begin{equation}
\label{lem:1oc:2}
\mathbb E\int_0^T\langle \mathcal D[u]_t-\mathcal D[v]_t ,u_t-v_t\rangle\, \mathrm d t\geq \delta\mathbb E\int_0^T\vert  u_t-v_t\vert^2\, \mathrm d t.
\end{equation}
\end{lemma}
One could refer to Ekeland and Temam \cite[Chapter 1, Proposition 5.5]{et76} for (i).
It is direct to obtain (ii) by (i) using the definition of uniform convexity.

When the cost functional is assumed to be uniformly convex, the optimal control should satisfy the stationary condition $\mathcal D[u]=0$. 
Therefore, combining the controlled SDE, the adjoint BSDE and the stationary condition, the following nonlinear Hamiltonian system is constructed to characterize the optimal control:
\begin{equation}\label{2:HSo}
\left\{
\begin{aligned}
&\mathrm d\bar X_t= \big[A(t)\bar X_t+B(t)\bar u_t+b(t)\big]\, \mathrm dt +\sum_{i=1}^d \big[C_i(t)\bar X_t +D_i(t)\bar u_t+\sigma_i(t)\big]\, \mathrm dW^i_t,\\
&\mathrm d \bar Y_t=-\left[A(t)^\top \bar Y_t +C(t)^\top \bar Z_t+\mathrm D_xl(t,\bar X_t,\bar u_t)\right]\, \mathrm dt+\sum_{i=1}^{d}\bar Z_t^i\, \mathrm dW_t^i,\\
&\mathrm D_ul(t,\bar X_t,\bar u_t)+B(t)^\top \bar Y_t+D(t)^\top\bar Z_t=0,\quad t\in [0,T],\\
& \bar X_0 = x, ~~~~~~~~\bar Y_T=\mathrm D_x g(\bar X_T).
\end{aligned}
\right.
\end{equation} 
\begin{remark}
Specially, let
\begin{equation}  \label{2:lqc}
\begin{aligned}
g(x)=&\frac12\langle Gx,x\rangle+\langle r,x\rangle,\\
l(t,x,u)=&\frac12\left\langle \begin{pmatrix}
Q(t)& S(t)^\top\\
S(t)& R(t)
\end{pmatrix}\begin{pmatrix}
x\\
u
\end{pmatrix},\begin{pmatrix}
x\\
u
\end{pmatrix}\right\rangle+\left\langle \begin{pmatrix}
q(t)\\
\rho(t)
\end{pmatrix},\begin{pmatrix}
x\\
u
\end{pmatrix}\right\rangle,
\end{aligned}
\end{equation}
where the random variables $G,r$ are $\mathcal F_T$-measurable, the stochastic processes $Q,S,R,q,\rho$ are $\mathbb F$-adapted, i.e., the LQ case.
It is obvious that
$$
\begin{aligned}
\mathrm D_x g(x)&=Gx+r,\\
\mathrm D_x l(t,x,u)&=Q(t)x+S(t)^\top u+q(t),\\
\mathrm D_u l(t,x,u)&=S(t)x+R(t)u+\rho(t).
\end{aligned}
$$
Accordingly, the above nonlinear Hamiltonian system \eqref{2:HSo} is reduced to the following linear one:
\begin{equation}\label{2:hslq}
\left\{
\begin{aligned}
&\mathrm d\bar X_s= \big[A(s)\bar X_s+B(s)\bar u_s+b(s)\big]\, \mathrm ds +\sum_{i=1}^d \big[C_i(s)\bar X_s +D_i(s)\bar u_s+\sigma_i(s)\big]\, \mathrm dW^i_s,\\
&\mathrm d \bar Y_s=-\left[A(s)^\top \bar Y_s +C(s)^\top \bar Z_s+Q(s)\bar X_s+S(s)^\top\bar u_s+q(s)\right]\, \mathrm ds+\sum_{i=1}^{d}\bar Z_s^{t,\xi,i}\, \mathrm dW_s^i,\\
& R(s)\bar u_s+S(s)\bar X_s+B(s)^\top \bar Y_s+D(s)^\top\bar Z_s+\rho(s)=0,\quad s\in [t,T],\\
& \bar X_t =\xi, ~~~~~~~~\bar Y_T=G\bar X_T+r.
\end{aligned}
\right.
\end{equation} 
The above Hamiltonian system \eqref{2:hslq} will be useful in the next section.
\end{remark}

Note that the Frech\'et derivative is an extension of gradients in infinite dimensional spaces.
Combining Proposition \ref{prop:fd} and Lemma \ref{lem:1oc}, we are able to generalize the gradient descent method to the Hilbert space $\mathscr U_{\mathrm{ad}}[0,T]$.
Consequently, the unique optimal control $\bar u$ is obtained and the well-posedness of the Hamiltonian system \eqref{2:HSo} is proven.

\begin{theorem}
\label{th:wp}
Let \ref{A1} and \ref{A2} hold and $J$ be uniformly convex. 
Then the Hamiltonian system \eqref{2:HSo} admits a unique solution $(\bar X,\bar Y,\bar Z,\bar u)\in\mathcal M[0,T]$. 
Moreover, $\bar u$ is the unique optimal control of Problem (LC)$_{(0,x)}$. 
Furthermore, there is a constant $K>0$ only depending on $T,\delta$, the Lipschitz constant of $\mathrm D_xg,\mathrm D_xl,\mathrm D_ul$ and the bound of $A,B,C,D$ such that
\begin{equation}\label{th:wp:1}
\mathbb E\bigg[\sup_{0\leq t\leq T}\vert \bar X_t\vert^2+ \sup_{0\leq t\leq T}\vert\bar Y_t\vert^2+\int_0^T\vert \bar Z_t\vert^2\, \mathrm dt+\int_0^T\vert \bar u_t\vert^2\, \mathrm dt\bigg]
\leq KI,
\end{equation} 
where 
\begin{equation*}
\begin{aligned}
I=&\vert  x\vert^2+\mathbb E\bigg[\vert  \mathrm D_xg(0)\vert^2 +\left(\int_0^T\vert  b(t)\vert\, \mathrm dt\right)^2+\left(\int_0^T\vert\mathrm D_xl(t,0,0)\vert\, \mathrm dt\right)^2\\
&+ \int_0^T [\vert  \sigma(t)\vert^2+\vert \mathrm D_u l(t,0,0)\vert^2]\, \mathrm dt\bigg].
\end{aligned}
\end{equation*} 
\end{theorem}
\begin{proof}
Firstly, we would like to prove the operator $\mathcal D$ is uniformly Lipschitz continuous with respect to $u$. 
By the standard estimate for BSDEs and \eqref{pf:prop:fd:2}, there exists a constant $K>0$ such that, for any $u,v\in\mathscr U_{\mathrm {ad}}[0,T]$, 
\begin{equation}
\mathbb E\bigg[\sup_{0\leq t\leq T}\vert Y^u_t-Y^v_t\vert^2+\int_0^T\vert Z^u_t-Z^v_t\vert^2\, \mathrm d t\bigg]\leq K\mathbb E\int_0^T\vert u_t-v_t\vert^2\, \mathrm d t.
\end{equation} 
Therefore, we obtain
\begin{equation}\label{pf:th:wp:1}
\begin{aligned}
\mathbb E \int_0^T\vert \mathcal D[u]_t-\mathcal D[v]_t\vert^2\, \mathrm d t \leq &K\mathbb E\int_0^T[\vert X^u_t-X^v_t\vert^2+\vert u_t-v_t\vert^2\\
&+\vert Y^u_t-Y^v_t\vert^2+\vert Z^u_t-Z^v_t\vert^2]\, \mathrm d t\\
\leq&K\mathbb E\int_0^T\vert u_t-v_t\vert^2\, \mathrm d t .
\end{aligned}
\end{equation} 

Secondly, we generalize the gradient descent method to the Hilbert space $\mathscr U_{\mathrm{ad}}[0,T]$ to find the optimal control.
Let $\eta>0$ be undetermined. 
For any $u\in\mathscr U_{\mathrm {ad}}[0,T]$, we construct a mapping
$$
\mathcal T[u]:=u -\eta \mathcal D[u],
$$
which maps $ \mathscr U_{\mathrm {ad}}[0,T]$ into itself. 
By Lemma \ref{lem:1oc} and \eqref{pf:th:wp:1}, we derive that, for any $u,v\in\mathscr U_{\mathrm {ad}}[0,T]$,
\begin{equation}\label{pf:th:wp:2}
\begin{aligned}
&\Vert\mathcal T[u]-\mathcal T[v]\Vert^2_{L^2}=\mathbb E\int_0^T\big\vert u_t-v_t-\eta (\mathcal D[u]_t-\mathcal D[v]_t)\big\vert^2\, \mathrm d t\\
=&\mathbb E\int_0^T\big[\vert u_t-v_t\vert^2-2\eta\langle \mathcal D[u]_t-\mathcal D[v]_t,u_t-v_t\rangle +\eta^2\vert \mathcal D[u]_t-\mathcal D[v]_t\vert^2\big]\, \mathrm d t \\
\leq&(1-2\eta\delta+\eta^2K)\Vert u-v\Vert^2_{L^2}.
\end{aligned}
\end{equation} 
Letting $\eta=\delta/K$ leads to $1-2\eta\delta+\eta^2K=1-\delta^2/K<1$. 
By Banach's contraction mapping principle, the mapping $\mathcal T$ admits a unique fixed point $\bar u$. 
By $\mathcal T(\bar u)=\bar u$, we obtain
$$
\mathcal D[\bar u]_t=\mathrm D_ul(t,\bar X_t,\bar u_t)+B(t)^\top \bar Y_t+D(t)^\top \bar Z_t=0,
$$
where $(\bar X,\bar Y,\bar Z):=(X^{\bar u},Y^{\bar u},Z^{\bar u})$. 
Then, the quadruple $(\bar X,\bar Y,\bar Z,\bar u)$ is the unique solution of the Hamiltonian system \eqref{2:HSo}. 

Thirdly, we would like to prove the estimate \eqref{th:wp:1}. 
By the standard estimates for SDEs and BSDEs, there exists a constant $K>0$ such that for any $u\in\mathscr U_{\mathrm{ad}}[0,T]$,
\begin{equation}\label{pf:th:wp:3}
\mathbb E\bigg[\sup_{0\leq t\leq T}\vert X_t^u\vert^2+ \sup_{0\leq t\leq T}\vert Y_t^u\vert^2+\int_0^T\vert Z_t^u\vert^2\, \mathrm dt\bigg]
\leq KI+K\mathbb E\int_0^T\vert u_t\vert^2\, \mathrm dt.
\end{equation} 
Let $(X^0,Y^0,Z^0)$ be the solution to the case of $u\equiv0$. 
Replacing $u$ and $v$ by $0$ and $\bar u$ in \eqref{prop:fd:2}, respectively, we deduce
\begin{equation*}
\begin{aligned}
 &\mathbb E\int_0^T\langle\mathcal D[0]_t,\bar u_t\rangle \mathrm d t\\
=&\mathbb E\left[\langle\mathrm D_xg(X^0_T),\bar X_T-X^0_T\rangle+\int_0^T[\langle\mathrm D_xl(t,X^0_t,0),\bar X_t-X^0_t\rangle+\langle\mathrm D_ul(t,X^0_t,0),\bar u_t\rangle] \mathrm d t\right].
\end{aligned}
\end{equation*}
Since $\mathcal D[\bar u]=0$, by Lemma \ref{lem:1oc}, we obtain
\begin{equation*}
\begin{aligned}
\delta \mathbb E\int_0^T\vert\bar u_t\vert^2\, \mathrm d t\leq&\mathbb E\int_0^T\langle\mathcal D[0]_t-\mathcal D[\bar u]_t,0-\bar u_t\rangle \mathrm d t\\
=&-\mathbb E\bigg[\langle\mathrm D_xg(X^0_T),\bar X_T-X^0_T\rangle+\int_0^T[\langle\mathrm D_xl(t,X^0_t,0),\bar X_t-X^0_t\rangle\\
&+\langle\mathrm D_ul(t,X^0_t,0),\bar u_t\rangle] \mathrm d t\bigg]\\
\leq&\mathbb E\Bigg[\left(\vert\mathrm D_xg(X^0_T)\vert+\int_0^T\vert\mathrm D_xl(t,X^0_t,0)\vert\, \mathrm d t\right)\left(\sup_{0\leq t\leq T}\vert \bar X_t-X^0_t\vert\right)\\
&+\left(\int_0^T\vert\mathrm D_ul(t,X^0_t,0)\vert^2\, \mathrm d t\right)^{1/2}\left(\int_0^T\vert\bar u_t\vert^2\, \mathrm d t\right)^{1/2}\Bigg].
\end{aligned}
\end{equation*}
By Young's inequality, \eqref{pf:prop:fd:2} and \eqref{pf:th:wp:3}, we deduce that, for any $\varepsilon>0$, there exists a constant $K_{\varepsilon}>0$ depending on $\varepsilon$ such that
\begin{equation*}
\delta \mathbb E\int_0^T\vert\bar u_t\vert^2\, \mathrm d t\leq\varepsilon\mathbb E\int_0^T\vert\bar u_t\vert^2\, \mathrm d t+K_{\varepsilon}I.
\end{equation*}
Taking $\varepsilon=\delta/2$, we get
\begin{equation*}
\mathbb E\int_0^T\vert\bar u_t\vert^2\, \mathrm d t\leq KI.
\end{equation*}
By \eqref{pf:th:wp:3} again, we conclude \eqref{th:wp:1}.

Finally, we prove that $\bar u$ is the unique optimal control. 
For any $v\in\mathscr U_{\mathrm {ad}}[0,T]$, we have
\begin{equation}
\begin{aligned}
J(v)-J(\bar u)=&\mathbb E\int_0^1\int_0^T\langle \mathcal D[\bar u+\gamma (v-\bar u)]_t,v_t-\bar u_t\rangle\, \mathrm d t \, \mathrm d\gamma\\
=&\mathbb E\int_0^1\int_0^T\langle \mathcal D[\bar u+\gamma (v-\bar u)]_t-\mathcal D[\bar u]_t,v_t-\bar u_t \rangle\, \mathrm d t \, \mathrm d\gamma.
\end{aligned}
\end{equation}
By Lemma \ref{lem:1oc}, we have 
\begin{equation}
J(v)-J(\bar u)\geq\delta\left(\int_0^1\gamma\, \mathrm d\gamma\right)\left(\mathbb E\int_0^T\vert v_t-\bar u_t\vert^2\, \mathrm d t\right) =\frac{\delta}2\mathbb E\int_0^T\vert v_t-\bar u_t\vert^2\, \mathrm d t,
\end{equation}
which shows that $\bar u$ is the unique optimal control. The proof is completed.
\end{proof}

It is natural to extend the initial pair $(t,x)\in[0,T]\times \mathbb R^n$ to $(t,\xi)\in[0,T]\times L^2_{\mathcal F_t}(\Omega;\mathbb R^n)$ with compatible notations for $J(t,\xi;u)$ and $X^{t,\xi,u}$.
In detail, for any given $(t,\xi)\in[0,T]\times L^2_{\mathcal F_t}(\Omega;\mathbb R^n)$ and $u\in\mathscr U_{\mathrm{ad}}[t,T]$, the cost functional is defined as
\begin{equation} 
J(t,\xi;u)=\mathbb E\left[g(X^{t,\xi,u}_T)+\int_t^Tl(s,X^{t,\xi,u}_s,u_s)\, \mathrm ds\right],
\end{equation} 
where $X^{t,\xi,u}$ satisfies the following linear SDE:
\begin{equation}
\left\{
\begin{aligned}
\mathrm dX^{t,\xi,u}_s=& \big[A(s)X^{t,\xi,u}_s+B(s)u_s+b(s)\big] \, \mathrm ds\\
&+\sum_{i=1}^d \big[C_i(s)X^{t,\xi,u}_s+D_i(s)u_s+\sigma_i(s)\big]\, \mathrm dW^i_s,\quad s\in [t,T],\\
X^{t,\xi,u}_t =& \xi.
\end{aligned}
\right.
\end{equation}
Correspondingly, it is straightforward to generalize Theorem \ref{th:wp} to the case where the initial pair $(t,\xi)\in[0,T]\times L^2_{\mathcal F_t}(\Omega;\mathbb R^n)$ using the same method. 
Specifically, if the mapping $u\mapsto J(t,\xi;u)$ is uniformly convex for the given initial pair $(t,\xi)\in[0,T]\times L^2_{\mathcal F_t}(\Omega;\mathbb R^n)$, the following Hamiltonian system admits a unique solution $(\bar X^{t,\xi},\bar Y^{t,\xi},\bar Z^{t,\xi},\bar u^{t,\xi})\in\mathcal M[t,T]$:

\begin{equation}\label{2:hst}
\left\{
\begin{aligned}
&\mathrm d\bar X_s^{t,\xi}= \big[A(s)\bar X_s^{t,\xi}+B(s)\bar u_s^{t,\xi}+b(s)\big]\, \mathrm ds +\sum_{i=1}^d \big[C_i(s)\bar X_s^{t,\xi} +D_i(s)\bar u_s^{t,\xi}+\sigma_i(s)\big]\, \mathrm dW^i_s,\\
&\mathrm d \bar Y_s^{t,\xi}=-\left[A(s)^\top \bar Y_s^{t,\xi} +C(s)^\top \bar Z_s^{t,\xi}+\mathrm D_xl(s,\bar X_s^{t,\xi},\bar u_s^{t,\xi})\right]\, \mathrm ds+\sum_{i=1}^{d}\bar Z_s^{t,\xi,i}\, \mathrm dW_s^i,\\
&\mathrm D_ul(s,\bar X_s^{t,\xi},\bar u_s^{t,\xi})+B(s)^\top \bar Y_s^{t,\xi}+D(s)^\top\bar Z_s^{t,\xi}=0,\quad s\in [t,T],\\
& \bar X_t^{t,\xi} =\xi, ~~~~~~~~\bar Y_T^{t,\xi}=\mathrm D_x g(\bar X_T^{t,\xi}).
\end{aligned}
\right.
\end{equation} 
Moreover, $\bar u^{t,\xi}$ is the unique optimal control such that 
$$
J(t,\xi;\bar u^{t,\xi})=\inf_{u\in\mathscr U_{\mathrm{ad}}[t,T]}J(t,\xi;u).
$$
Furthermore, there exists a constant $K>0$ only depending on $T,\delta$, the Lipschitz constant of $\mathrm D_xg,\mathrm D_xl,\mathrm D_ul$ and the bound of $A,B,C,D$ such that
\begin{equation}\label{2:est}
\mathbb E\bigg[\sup_{t\leq s\leq T}\vert \bar X_s^{t,\xi}\vert^2+ \sup_{t\leq s\leq T}\vert\bar Y_s^{t,\xi}\vert^2+\int_t^T\vert \bar Z_s^{t,\xi}\vert^2\, \mathrm dt+\int_t^T\vert \bar u_s^{t,\xi}\vert^2\, \mathrm ds\bigg]
\leq KI^{t,\xi},
\end{equation} 
where
$$
\begin{aligned}
I^{t,\xi}=&\mathbb E\bigg[\vert\xi\vert^2+\vert\mathrm D_xg(0)\vert^2 +\left(\int_t^T\vert  b(s)\vert\, \mathrm ds\right)^2+\left(\int_t^T\vert\mathrm D_xl(s,0,0)\vert\, \mathrm ds\right)^2\\
&+ \int_t^T [\vert  \sigma(s)\vert^2+\vert \mathrm D_u l(s,0,0)\vert^2]\, \mathrm ds\bigg].
\end{aligned}
$$
\section{Regularity of solutions to Hamiltonian systems}\label{sec3}
This section is devoted to establishing the regularity of the solution $(\bar X^{t,\xi},\bar Y^{t,\xi},\bar Z^{t,\xi},\bar u^{t,\xi})$, including the derivatives with respect to $x$ and the Malliavin derivatives.
In addition to \eqref{1:do} and \eqref{1:do2}, we further introduce the following differential operators:
$$
\mathrm D^2_{ux}=\mathrm D_u\mathrm D_x^\top,~~ \mathrm D^2_{xu}:=\mathrm D_x\mathrm D_u^\top.
$$
From now on, we would like to suppose the following assumptions:
\begin{assumption}\label{A3}
All coefficients $A,B,C,D,b,\sigma,g,l$ are deterministic.
The derivatives $\mathrm D_xg,\mathrm D^2_{xx}g,\mathrm D_xl,\mathrm D_ul,\mathrm D^2_{xx}l,$ $\mathrm D^2_{xu}l,\mathrm D^2_{ux}l,\mathrm D^2_{uu}l$ exist.
The functions $A,B,C,D,b,\sigma,g,l$ and all derivatives of $g,l$ are jointly continuous in $(t,x,u)$.
All second-order derivatives of $g,l$ are uniformly bounded.
\end{assumption}
\begin{assumption}\label{A4}
  There exists a constant $\delta>0$ such that, for any $(t,\xi)\in[0,T]\times L^2_{\mathcal F_t}(\Omega;\mathbb R^n)$, the mapping $u\mapsto J(t,\xi;u)-\frac\delta 2\Vert u\Vert^2_{L^2}$ is convex.
\end{assumption}

Under \ref{A3} and \ref{A4}, all assumptions of Theorem \ref{th:wp} are satisfied for any given initial pair $(t,\xi)\in[0,T]\times L^2_{\mathcal F_t}(\Omega;\mathbb R^n)$.
Therefore, the corresponding Hamiltonian system admits a unique solution $(\bar X^{t,\xi},\bar Y^{t,\xi},\bar Z^{t,\xi},\bar u^{t,\xi})\in\mathcal M[t,T]$.
\begin{remark}\label{rmk}
  Let $\mathbb F^t:=\{\mathcal F_s^t\}_{s\in[t,T]}$ be a completed filtration, where $\mathcal F_s^t$ is $\sigma(W_r-W_t,r\in[t,s])$ augmented by all $\mathbb P$-null sets. 
  Noting that all coefficients are deterministic, for any initial pair $(t,x)\in[0,T]\times \mathbb R^n$, we would like to consider the Hamiltonian system \eqref{2:hst} on the filtration $\mathbb F^t$.
  Thus, applying Theorem \ref{th:wp}, we derive that the unique solution $(\bar X^{t,x},\bar Y^{t,x},\bar Z^{t,x},\bar u^{t,x})$ is adapted to $\mathbb F^t$.
  In other words, $(\bar X^{t,x},\bar Y^{t,x},\bar Z^{t,x},\bar u^{t,x})$ is independent of $\mathcal F_t$.
  Specially, $\bar Y^{t,x}_t$ is deterministic.
\end{remark}

For any $(t,x)\in[0,T]\times\mathbb R^n$, we extend the solution $(\bar X^{t,x},\bar Y^{t,x},\bar Z^{t,x},\bar u^{t,x})$ from $\mathcal M[t,T]$ to $\mathcal M[0,T]$ by letting
\begin{equation*}
(\bar X^{t,x}_s,\bar Y^{t,x}_s,\bar Z^{t,x}_s,\bar u^{t,x}_s)=(x,\bar Y^{t,x}_t,0,0),~~\mbox{for all }s \in[0,t).
\end{equation*}
The following proposition comes from the estimate \eqref{2:est}.
\begin{proposition}\label{prop:est}
  Under Assumptions \ref{A3} and \ref{A4}, there exists a constant $K>0$ such that, for any $(t,x),(t',x')\in[0,T]\times \mathbb R^n$, 
\begin{equation}\label{prop:est:1}
\mathbb E\bigg[\sup_{0\leq s\leq T}(\vert \bar X^{t,x}_s \vert^2+\vert \bar Y^{t,x}_s \vert^2)+\int_0^T(\vert \bar Z^{t,x}_s \vert^2+\vert \bar u^{t,x}_s \vert^2)\, \mathrm ds\bigg]\leq K(1+\vert x\vert^2),
\end{equation} 
\begin{equation}\label{prop:est:2}
\begin{aligned}
&\mathbb E\bigg[\sup_{0\leq s\leq T}(\vert \bar X^{t,x}_s-\bar X^{t',x'}_s \vert^2+\vert \bar Y^{t,x}_s-\bar Y^{t',x'}_s\vert^2)\\
&+\int_0^T(\vert \bar Z^{t,x}_s-\bar Z^{t',x'}_s \vert^2+\vert \bar u^{t,x}_s-\bar u^{t',x'}_s \vert^2)\, \mathrm ds\bigg]\\
\leq &K\Big[\vert x-x'\vert^2+(1+\vert x\vert^2+\vert x'\vert^2)\vert t-t'\vert\Big].
\end{aligned}
\end{equation} 
\end{proposition}

\begin{proof}
The estimate \eqref{prop:est:1} comes from the estimate \eqref{2:est} immediately. 
To prove \eqref{prop:est:2}, we assume $t\leq t'$ without loss of generality. 
Denote $\hat X:=\bar X^{t,x}-\bar X^{t',x'}$, similarly for $\hat Y,\hat Z$ and $\hat u$.
It is straightforward to check that $(\hat X,\hat Y,\hat Z,\hat u)$ satisfies the following nonlinear Hamiltonian system:
\begin{equation}\label{pf:prop:est:1}
\left\{
\begin{aligned}
&\mathrm d\hat X_s = \big\{A(s)\hat X_s+B(s)\hat u_s+\hat b(s)\big\} \, \mathrm ds+\sum_{j=1}^d \big\{C_j(s)\hat X_s+D_j(s)\hat u_s+\hat\sigma_j(s)\big\}\, \mathrm dW^j_s,\\
&\mathrm d \hat Y_s=-\big\{A(s)^\top \hat Y_s +C(s)^\top \hat Z_s+\mathrm D_x\hat l(s,\hat X_s,\hat u_s)\big\}\, \mathrm ds+\sum_{j=1}^{d}\hat Z^{j}_s\, \mathrm dW^j_s,\\
&\mathrm D_u\hat l(s,\hat X_s,\hat u_s)+B(s)^\top\hat Y_s+D(s)^\top\hat Z_s=0,\quad s\in [t,T],\\
&\hat X_t = x-x', ~~~~~\hat Y^h_T=\mathrm D_x\hat g(\hat X_T),
\end{aligned}
\right.
\end{equation} 
where 
\begin{equation*}
\begin{aligned}
\hat b(s)=&[A(s)x'+b(s)]\mathbbm 1_{[t,t']}(s),~~\hat \sigma_i(s)=[C_j(s)x'+\sigma_j(s)]\mathbbm 1_{[t,t']}(s),\\
\hat l(s,x,u)=&l(s,x+\bar X^{t',x'}_s,u+\bar u^{t',x'}_s)\\
&-\langle \mathrm D_ul(s,\bar X^{t',x'}_s,\bar u^{t',x'}_s)-[B(s)^\top Y^{t',x'}_{t'}+\mathrm D_ul(s,x',0)]\mathbbm 1_{[t,t']}(s),u\rangle\\
&-\langle \mathrm D_xl(s,\bar X^{t',x'}_s,\bar u^{t',x'}_s)-[A(s)^\top\bar Y^{t',x'}_{t'}+\mathrm D_xl(s,x',0)]\mathbbm 1_{[t,t']}(s),x\rangle,\\
\hat g(x)=&g(x+\bar X^{t',x'}_T)-\langle\mathrm D_xg(\bar X^{t',x'}_T),x\rangle,
\end{aligned}
\end{equation*} 
and $\mathbbm 1_{[t,t']}$ is the indicator function of $[t,t']$.
Then, we introduce the following cost functional:
\begin{equation} 
\hat J(t,x-x';u)=\mathbb E\left[\hat g(\hat X^{t,x,u}_T)+\int_t^T\hat l(s,\hat X^{t,x,u}_s,u_s)\, \mathrm ds\right],
\end{equation} 
where $\hat X^{t,x,u}$ satisfies the following linear SDE:
\begin{equation}
\left\{
\begin{aligned}
\mathrm d\hat X^{t,x,u}_s=& \big[A(s)\hat X^{t,x,u}_s+B(s)u_s+\hat b(s)\big] \, \mathrm ds\\
&+\sum_{i=1}^d \big[C_i(s)\hat X^{t,x,u}_s+D_i(s)u_s+\hat \sigma_i(s)\big]\, \mathrm dW^i_s,\quad s\in [t,T],\\
\hat X^{t,x,u}_t =& x-x'.
\end{aligned}
\right.
\end{equation}
Noting that $\hat X^{t,x,u}+\bar X^{t',x'}=X^{t,x,u+\bar u^{t',x'}}$, we have
\begin{equation}
\begin{aligned}
&\hat J(t,x-x';u)=J(t,x;u+\bar u^{t',x'})-\mathbb E\Bigg[\langle\mathrm D_xg(\bar X^{t',x'}_T),\hat X^{t,x,u}_T\rangle\\
&+\int_t^T\left\langle
  \begin{pmatrix}
\mathrm D_xl(s,\bar X^{t',x'}_s,\bar u^{t',x'}_s)-[A(s)^\top\bar Y^{t',x'}_{t'}+\mathrm D_xl(s,x',0)]\mathbbm 1_{[t,t']}(s)\\
\mathrm D_ul(s,\bar X^{t',x'}_s,\bar u^{t',x'}_s)-[B(s)^\top Y^{t',x'}_{t'}+\mathrm D_ul(s,x',0)]\mathbbm 1_{[t,t']}(s)
\end{pmatrix}
,\begin{pmatrix}
\hat X^{t,x,u}_s\\
  u_s
\end{pmatrix}
\right\rangle\, \mathrm ds\Bigg].
\end{aligned}
\end{equation}
Since $J(t,x;u)$ is uniformly convex in $u$, we conclude $J(t,x;u+\bar u^{t',x'})$, which composites a linear transformation, is also uniformly convex in $u$.
The rest parts are linear in $u$.
Thus, $\hat J$ is uniformly convex, showing that the Hamiltonian system \eqref{pf:prop:est:1} satisfies the assumptions in Theorem \ref{th:wp}.
By the estimate \eqref{2:est} again, we derive \eqref{prop:est:2}.
\end{proof}

For any given initial pair $(t,\xi)\in[0,T]\times L^2_{\mathcal F_t}(\Omega;\mathbb R^n)$ and any admissible control $u\in \mathscr U_{\mathrm {ad}}[t,T]$, we introduce the following notations:
for any $s\in[t,T]$,
\begin{equation}\label{3:not1}
\begin{aligned}
\mathcal G^{t,\xi,u}:=\mathrm D^2_{xx}g(X^{t,\xi,u}_T),~~
\begin{pmatrix}
\mathcal Q^{t,\xi,u}_s& (\mathcal S^{t,\xi,u}_s)^\top\\
\mathcal S^{t,\xi,u}_s& \mathcal R^{t,\xi,u}_s
\end{pmatrix}:=\begin{pmatrix}
\mathrm D^2_{xx}l& \mathrm D^2_{xu}l\\
\mathrm D^2_{ux}l& \mathrm D^2_{uu}l
\end{pmatrix}(s,X^{t,\xi,u}_s,u_s).
\end{aligned}
\end{equation}

Using the above notations, we provide the following lemma.
Compared with Lemma \ref{lem:1oc}, which is the first-order condition for the uniform convexity, the following lemma can be regarded as a second-order condition.

\begin{lemma}\label{lem:2oc}
  Let Assumption \ref{A3} hold. Then for any given initial pair $(t,\xi)\in[0,T]\times L^2_{\mathcal F_t}(\Omega;\mathbb R^n)$,
  the mapping $u\mapsto J(t,\xi;u)$ is uniformly convex if and only if there exists a constant $\delta>0$, such that for any $u,v\in \mathscr U_{\mathrm {ad}}[t,T]$,
\begin{equation}\label{lem:2oc:1}
\begin{aligned}
\mathbb E\int_t^T \left\langle \begin{pmatrix}
\mathcal Q^{t,\xi,u}_s& (\mathcal S^{t,\xi,u}_s)^\top\\
\mathcal S^{t,\xi,u}_s& \mathcal R^{t,\xi,u}_s
\end{pmatrix}\begin{pmatrix}
X^{t,\xi,v}_s-X^{t,\xi,u}_s\\
v_s-u_s
\end{pmatrix},\begin{pmatrix}
X^{t,\xi,v}_s-X^{t,\xi,u}_s\\
v_s-u_s
\end{pmatrix}\right\rangle \,\mathrm ds\\
+\mathbb E[\langle \mathcal G^{t,\xi,u}(X^{t,\xi,v}_T-X^{t,\xi,u}_T),X^{t,\xi,v}_T-X^{t,\xi,u}_T\rangle]\geq\delta \mathbb E\int_t^T\vert v_s-u_s\vert^2 \mathrm ds.
\end{aligned}
\end{equation}
\end{lemma}
\begin{proof}
  For the sufficiency, we assume that $J$ is uniformly convex.
  Substituting \eqref{prop:fd:2} into \eqref{lem:1oc:2}, we have, for any $ u,v\in \mathscr U_{\mathrm {ad}}[t,T]$,
  \begin{equation}\label{pf:lem:2oc:1}
\begin{aligned}
&\mathbb E\int_t^T \left\langle 
\begin{pmatrix}
\mathrm D_{x}l(s,X^{t,\xi,v}_s,v_s)-\mathrm D_{x}l(s,X^{t,\xi,u}_s,u_s)\\
\mathrm D_{u}l(s,X^{t,\xi,v}_s,v_s)-\mathrm D_{u}l(s,X^{t,\xi,u}_s,u_s)
\end{pmatrix},
\begin{pmatrix}
X^{t,\xi,v}_s-X^{t,\xi,u}_s\\
v_s-u_s
\end{pmatrix}
\right\rangle\,\mathrm ds\\
&+\mathbb E[\langle \mathrm D_{x}g(X^{t,\xi,v}_T)-\mathrm D_{x}g(X^{t,\xi,u}_T),X^{t,\xi,v}_T-X^{t,\xi,u}_T\rangle]\geq\delta \mathbb E\int_t^T\vert v_s-u_s\vert^2 \mathrm ds.
\end{aligned}
\end{equation}
For simplification, we denote $\hat u:=v-u$ and $\hat X^{v-u}:=X^{t,\xi,v}-X^{t,\xi,u}$.
Let $u^{\varepsilon}:=u+\varepsilon \hat u$, where $\varepsilon\in(0,1]$. 
By the linearity of the controlled SDE, we have $X^{t,\xi,u^{\varepsilon}}-X^{t,\xi,u}=\varepsilon \hat X^{v-u}$ for any $\varepsilon\in(0,1]$.
Therefore, replacing $v$ by $u^{\varepsilon}$ in \eqref{pf:lem:2oc:1}, we derive
\begin{equation}\label{pf:lem:2oc:2}
\begin{aligned}
\varepsilon^2\mathbb E\int_t^T \left\langle 
\begin{pmatrix}
\mathcal Q^{t,\xi,u}_s& (\mathcal S^{t,\xi,u}_s)^\top\\
\mathcal S^{t,\xi,u}_s& \mathcal R^{t,\xi,u}_s
\end{pmatrix}
\begin{pmatrix}
\hat X^{v-u}_s\\
\hat u_s
\end{pmatrix},\begin{pmatrix}
\hat X^{v-u}_s\\
\hat u_s
\end{pmatrix}\right\rangle \,\mathrm ds\\
+\varepsilon^2\mathbb E[\langle \mathcal G^{t,\xi,u}\hat X^{v-u}_T,\hat X^{v-u}_T\rangle]+\varepsilon^2\Gamma^{\varepsilon}\geq\delta \varepsilon^2\mathbb E\int_t^T\vert \hat u_s\vert^2 \mathrm ds,
\end{aligned}
\end{equation}
where
\begin{equation*}
\begin{aligned}
&\Gamma^{\varepsilon}=\mathbb E\int_0^1\langle [\mathcal G^{t,\xi,u^{\gamma\varepsilon}}-\mathcal G^{t,\xi,u}]\hat X^{v-u}_T,\hat X^{v-u}_T\rangle\,\mathrm d\gamma\\
&+\mathbb E\int_t^T \int_0^1\Bigg\langle
\begin{pmatrix}
\mathcal Q^{t,\xi,u^{\gamma\varepsilon}}_s-\mathcal Q^{t,\xi,u}_s& (\mathcal S^{t,\xi,u^{\gamma\varepsilon}}_s-\mathcal S^{t,\xi,u}_s)^\top\\
\mathcal S^{t,\xi,u^{\gamma\varepsilon}}_s-\mathcal S^{t,\xi,u}_s& \mathcal R^{t,\xi,u^{\gamma\varepsilon}}_s-\mathcal R^{t,\xi,u}_s
\end{pmatrix}\begin{pmatrix}
\hat X^{v-u}_s\\
\hat u_s
\end{pmatrix},
\begin{pmatrix}
\hat X^{v-u}_s\\
\hat u_s
\end{pmatrix}\Bigg\rangle \,\mathrm d\gamma\,\mathrm ds
\end{aligned}
\end{equation*}
and $u^{\gamma\varepsilon}:=u+\gamma\varepsilon\hat u$.
Noting that $X^{t,\xi,u}+\gamma\varepsilon\hat X^{v-u}=X^{t,\xi,u^{\gamma\varepsilon}}$, the above representation follows.
Since all second-order derivatives of $g,l$ are bounded and continuous and 
$$
\mathbb E\bigg[\sup_{t\leq s\leq T}\vert \hat X^{v-u}_s\vert^2\bigg]\leq K\mathbb E\int_t^T\vert \hat u_s\vert^2 \mathrm ds,
$$
applying Lebesgue's dominated convergence theorem leads to
\begin{equation}
  \lim_{\varepsilon\to0}\Gamma^{\varepsilon}=0.
\end{equation}
Therefore, dividing both sides of \eqref{pf:lem:2oc:2} by $\varepsilon^2$ and letting $\varepsilon\to0$, we conclude \eqref{lem:2oc:1}.

For the necessary, we assume \eqref{lem:2oc:1} holds for any $u,v\in \mathscr U_{\mathrm {ad}}[t,T]$.
Then, noting that the equality \eqref{prop:fd:2} does not rely on the uniform convexity of $J$, we have
  \begin{equation}
\begin{aligned}
&\mathbb E\int_0^T\langle \mathcal D[u]_t-\mathcal D[v]_t ,u_t-v_t\rangle\, \mathrm d t=\mathbb E\int_0^1\langle \mathcal G^{t,\xi,u^{\varepsilon}}\hat X^{v-u}_T,\hat X^{v-u}_T\rangle\,\mathrm d\gamma\\
&+\mathbb E\int_t^T \int_0^1\Bigg\langle
\begin{pmatrix}
\mathcal Q^{t,\xi,u^{\varepsilon}}_s& (\mathcal S^{t,\xi,u^{\varepsilon}}_s)^\top\\
\mathcal S^{t,\xi,u^{\varepsilon}}_s& \mathcal R^{t,\xi,u^{\varepsilon}}_s
\end{pmatrix}\begin{pmatrix}
\hat X^{v-u}_s\\
\hat u_s
\end{pmatrix},
\begin{pmatrix}
\hat X^{v-u}_s\\
\hat u_s
\end{pmatrix}\Bigg\rangle \,\mathrm d\varepsilon\,\mathrm ds\geq\delta \mathbb E\int_t^T\vert v_s-u_s\vert^2\,\mathrm ds.
\end{aligned}
\end{equation}
By Lemma \ref{lem:1oc}, the cost functional $J$ is uniformly convex.
The proof is completed.
\end{proof}
\begin{remark}
When \eqref{2:lqc} holds, i.e., the LQ case, all second-order derivatives of $g,l$ are independent of $(x,u)$.
The condition \eqref{lem:2oc:1} is reduced to 
\begin{equation}\label{3:lquc}
\begin{aligned}
&\mathbb E\int_t^T\left\langle \begin{pmatrix}
Q(s)& S(s)^\top\\
S(s)& R(s)
\end{pmatrix}\begin{pmatrix}
\hat X_s\\
\hat u_s
\end{pmatrix},\begin{pmatrix}
\hat X_s\\
\hat u_s
\end{pmatrix}\right\rangle\, \mathrm ds+\mathbb E[ \langle G\hat X_T,\hat X_T\rangle]\geq\delta \mathbb E\int_t^T\vert u_s\vert^2\, \mathrm ds,
\end{aligned}
\end{equation} 
where we denote $\hat X:= X^{t,\xi,v}-X^{t,\xi,u}$ and $\hat u:=v-u$.
Noting that $\hat X$ satisfies the following linear controlled SDE:
\begin{equation}
\hat X_s= \int_t^s\big[A(s)\hat X_s+B(s) \hat u_s\big]\, \mathrm ds+\sum_{i=1}^d \int_t^s\big[C_i(s)\hat X_s+D_i(s) \hat u_s\big]\, \mathrm dW^i_s,~~~s\in[0,T],
\end{equation}
which is independent of $\xi$, the above condition \eqref{3:lquc} does not actually rely on the initial value.
Moreover, Sun et al. \cite[Proposition 5.4]{sxy21} showed that \eqref{3:lquc} holds for any $t\in[0,T]$ if it holds for $t=0$.
In other words, Assumption \ref{A4} is equivalent to that \eqref{3:lquc} holds for $t=0$ in the LQ case.

However, when $g,l$ are not quadratic, the second-order derivatives depend on $(x,u)$ in general.
Thus, the condition \eqref{lem:2oc:1} relies on the initial pair $(t,\xi)$, which makes it impossible to obtain a similar simplified condition.
\end{remark}
Based on the above preparations, we now consider the derivative of $(\bar X^{t,x},\bar Y^{t,x},\bar Z^{t,x},\bar u^{t,x})$ with respect to $x$.
For any $i\in\{1,...,n\}$ and any $h\in\mathbb R/\{0\}$, we denote $\Delta_i^hX^{t,x}_s=h^{-1}[\bar X^{t,x+he_i}_s-\bar X^{t,x}_s]$, similarly for $\Delta_i^hY^{t,x},\Delta_i^hZ^{t,x},\Delta_i^hu^{t,x}$. 
It is straightforward to verify that $(\Delta_i^hX^{t,x},\Delta_i^hY^{t,x},\Delta_i^hZ^{t,x},\Delta_i^hu^{t,x})$ satisfies the following linear Hamiltonian system:

\begin{equation}\label{3:lhsd}
\left\{
\begin{aligned}
&\mathrm d\Delta_i^hX^{t,x}_s = \big[A(s)\Delta_i^hX^{t,x}_s+B(s)\Delta_i^hu^{t,x}_s\big]\, \mathrm ds\\
&\qquad\qquad+\sum_{j=1}^d \big[C_j(s)\Delta_i^hX^{t,x}_s+D_j(s)\Delta_i^hu^{t,x}_s\big]\, \mathrm dW^j_s,\\
&\mathrm d \Delta_i^hY^{t,x}_s=-\big[A(s)^\top \Delta_i^hY^{t,x}_s +C(s)^\top \Delta_i^hZ^{t,x}_s+\mathcal Q^{t,x,h}_s\Delta_i^hX^{t,x}_s\\
&\qquad\qquad+(\mathcal S^{t,x,h}_s)^\top\Delta_i^hu^{t,x}_s\big]\, \mathrm ds+\sum_{j=1}^{d}\Delta_i^hZ^{t,x,j}_s\, \mathrm dW^j_s,\\
&\mathcal R^{t,x,h}_s\Delta_i^hu^{t,x}_s+\mathcal S^{t,x,h}_s\Delta_i^hX^{t,x}_s+B(s)^\top \Delta_i^hY^{t,x}_s+D(s)^\top\Delta_i^hZ^{t,x}_s=0,\quad s\in [t,T],\\
& \Delta_i^hX^{t,x}_t = e_i, ~~~~~~~~\Delta_i^hY^{t,x}_T=\mathcal G^{t,x,h} \Delta_i^hX^{t,x}_T,
\end{aligned}
\right.
\end{equation} 
where
\begin{equation*}
\begin{aligned}
&\begin{pmatrix}
\mathcal Q^{t,x,h}_s & (\mathcal S^{t,x,h}_s)^\top \\
\mathcal S^{t,x,h}_s & \mathcal R^{t,x,h}_s
\end{pmatrix}:=\int_0^1\begin{pmatrix}
\mathrm D^2_{xx}l& \mathrm D^2_{xu}l\\
\mathrm D^2_{ux}l & \mathrm D^2_{uu}l
\end{pmatrix}(s,\bar X^{t,x}_s+\lambda h\Delta_i^hX^{t,x}_s,\bar u^{t,x}_s+\lambda h\Delta_i^hu^{t,x}_s)\, \mathrm d\lambda,\\
&\mathcal G^{t,x,h}:=\int_0^1\mathrm D^2_{xx}g(\bar X^{t,x}_T+\lambda h\Delta_i^hX^{t,x}_T)\, \mathrm d\lambda.
\end{aligned}
\end{equation*} 
By the linearity of controlled SDE \eqref{1:SDEt}, we deduce
\begin{equation}
\begin{aligned}
\bar X^{t,x}+\lambda h\Delta_i^hX^{t,x}=&X^{t,x,\bar u^{t,x}}+\lambda (X^{t,x+he_i,\bar u^{t,x+he_i}}-X^{t,x,\bar u^{t,x}})\\
=&X^{t,x+\lambda he_i,\bar u^{t,x}+\lambda h\Delta_i^hu^{t,x}_s}.
\end{aligned}
\end{equation} 
Therefore, by Lemma \ref{lem:2oc}, for any $u,v\in\mathscr U_{\mathrm {ad}}[t,T]$,
\begin{equation}
\begin{aligned}
\mathbb E\int_t^T \left\langle \begin{pmatrix}
\mathcal Q^{t,x,h}_s& (\mathcal S^{t,x,h}_s)^\top\\
\mathcal S^{t,x,h}_s& \mathcal R^{t,x,h}_s
\end{pmatrix}\begin{pmatrix}
X^{t,x,v}_s-X^{t,x,u}_s\\
v_s-u_s
\end{pmatrix},\begin{pmatrix}
X^{t,x,v}_s-X^{t,x,u}_s\\
v_s-u_s
\end{pmatrix}\right\rangle \,\mathrm ds\\
+\mathbb E[\langle \mathcal G^{t,x,h}(X^{t,x,v}_T-X^{t,x,u}_T),X^{t,x,v}_T-X^{t,x,u}_T\rangle]\geq\delta \mathbb E\int_t^T\vert v_s-u_s\vert^2 \mathrm ds.
\end{aligned}
\end{equation}
Compared to the uniform convexity condition \eqref{3:lquc} of the LQ case, the Hamiltonian system \eqref{3:lhsd} satisfies the assumptions of Theorem \ref{th:wp}. 
By the estimate \eqref {2:est}, there exists a constant $K>0$ such that for any given $h\in\mathbb R/\{0\}$,
\begin{equation}
\mathbb E\bigg[\sup_{t\leq s\leq T}\big(\vert\Delta_i^hX^{t,x}_s\vert^2+\vert\Delta_i^hY^{t,x}_s\vert^2\big)+\int_t^T(\vert\Delta_i^hZ^{t,x}_s\vert^2+\vert\Delta_i^hu^{t,x}_s\vert^2)\, \mathrm ds\bigg]\leq K.
\end{equation}

The following linear Hamiltonian system will be proven to be the limit of the Hamiltonian system \eqref{3:lhsd} when $h\to0$.
\begin{equation}\label{3:lhsg}
\left\{
\begin{aligned}
&\mathrm d\nabla_i X^{t,x}_s = \big[A(s)\nabla_i X^{t,x}_s+B(s)\nabla_i u^{t,x}_s\big]\, \mathrm ds\\
&\qquad\qquad+\sum_{j=1}^d \big[C_j(s)\nabla_i X^{t,x}_s+D_j(s)\nabla_i u^{t,x}_s\big]\, \mathrm dW^j_s,\\
&\mathrm d \nabla_i Y^{t,x}_s=-\big[A(s)^\top \nabla_i Y^{t,x}_s +C(s)^\top \nabla_i Z^{t,x}_s+\mathcal Q_s^{t,x}\nabla_i X^{t,x}_s\\
&\qquad\qquad+(\mathcal S_s^{t,x})^\top\nabla_i u^{t,x}_s\big]\, \mathrm ds+\sum_{j=1}^{d}\nabla_i Z^{t,x,j}_s\, \mathrm dW^j_s,\\
&\mathcal R_s^{t,x}\nabla_i u^{t,x}_s+\mathcal S_s^{t,x}\nabla_i X^{t,x}_s+B(s)^\top \nabla_i Y^{t,x}_s+D(s)^\top\nabla_i Z^{t,x}_s=0,\quad s\in [t,T],\\
& \nabla_i X^{t,x}_t = e_i, ~~~~~~~~\nabla_i Y^{t,x}_T=\mathcal G \nabla X^{t,x}_T,
\end{aligned}
\right.
\end{equation} 
where
\begin{equation}\label{3:not2}
\begin{pmatrix}
\mathcal Q^{t,x} & (\mathcal S^{t,x})^\top \\
\mathcal S^{t,x} & \mathcal R^{t,x}
\end{pmatrix}:=\begin{pmatrix}
\mathcal Q^{t,x,\bar u^{t,x}}& (\mathcal S^{t,x,\bar u^{t,x}})^\top\\
\mathcal S^{t,x,\bar u^{t,x}}& \mathcal R^{t,x,\bar u^{t,x}}
\end{pmatrix}, ~~~\mathcal G^{t,x}:=\mathcal G^{t,x,\bar u^{t,x}}.
\end{equation} 
The Hamiltonian system \eqref{3:lhsg} admits a unique solution $(\nabla_i X^{t,x},\nabla_i Y^{t,x},\nabla_i Z^{t,x},\nabla_i u^{t,x})\in\mathcal M[t,T]$ when applying Theorem \ref{th:wp} and Lemma \ref{lem:2oc} again. 
Moreover, there exists a constant $K>0$ such that
\begin{equation}\label{3:lhsg:est}
\mathbb E\bigg[\sup_{t\leq s\leq T}\big(\vert\nabla_i X^{t,x}_s\vert^2+\vert\nabla_i Y^{t,x}_s\vert^2\big)+\int_t^T(\vert\nabla_i Z^{t,x}_s\vert^2+\vert\nabla_i u^{t,x}_s\vert^2)\, \mathrm ds\bigg]\leq K.
\end{equation}

Next, we provide the convergence of $(\Delta_i^h X^{t,x},\Delta_i^h Y^{t,x},\Delta_i^h Z^{t,x},\Delta_i^h u^{t,x})$ and the continuous dependence with respect to $(t,x)$ of the derivatives $(\nabla_i X^{t,x},\nabla_i Y^{t,x},\nabla_i Z^{t,x},\nabla_i u^{t,x})$.
\begin{proposition}\label{prop:cvg}
Under Assumptions \ref{A3} and \ref{A4}, for any $(t,x)\in[0,T]\times\mathbb R^n$ and any $i\in\{1,...,n\}$, we have
\begin{equation}\label{prop:cvg:1}
\begin{aligned}
\lim_{h\to0}\mathbb E\bigg[\sup_{t\leq s\leq T}\left[\vert\Delta^h_iX^{t,x}_s-\nabla_i X^{t,x}_s\vert^2+\vert\Delta^h_iY^{t,x}_s-\nabla_i Y^{t,x}_s\vert^2\right]+\\
\int_t^T(\vert\Delta^h_iZ^{t,x}_s-\nabla_i Z^{t,x}_s\vert^2+\vert\Delta^h_iu^{t,x}_s-\nabla_i u^{t,x}_s\vert^2)\, \mathrm ds\bigg]=0
\end{aligned}
\end{equation}
and
\begin{equation}\label{prop:cvg:2}
\begin{aligned}
\lim_{(t',x')\to (t,x)}\mathbb E\bigg[\sup_{0\leq s\leq T}\left[\vert\nabla_iX^{t',x'}_s-\nabla_i X^{t,x}_s\vert^2+\vert\nabla_iY^{t',x'}_s-\nabla_i Y^{t,x}_s\vert^2\right]+\\
\int_0^T(\vert\nabla_iZ^{t',x'}_s-\nabla_i Z^{t,x}_s\vert^2+\vert\nabla_iu^{t',x'}_s-\nabla_i u^{t,x}_s\vert^2)\, \mathrm ds\bigg]=0.
\end{aligned}
\end{equation}
\end{proposition}
\begin{proof}
To prove \eqref{prop:cvg:1}, we denote $\hat X^h:=\Delta^h_iX^{t,x}-\nabla_i X^{t,x}$, similarly for $\hat Y^h,\hat Z^h,\hat u^h$. 
Then, it is clear that $(\hat X^h,\hat Y^h,\hat Z^h,\hat u^h)$ satisfies the following linear Hamiltonian system, which admits a similar form to \eqref{3:lhsd}:

\begin{equation}\label{pf:prop:cvg:1}
\left\{
\begin{aligned}
&\mathrm d\hat X^h_s = \big[A(s)\hat X^h_s+B(s)\hat u^h_s\big]\, \mathrm ds+\sum_{j=1}^d \big[C_j(s)\hat X^h_s+D_j(s)\hat u^h_s\big]\, \mathrm dW^j_s,\\
&\mathrm d \hat Y^h_s=-\big[A(s)^\top \hat Y^h_s +C(s)^\top \hat Z^h_s+\mathcal Q^{t,x,h}_s\hat X^h_s+(\mathcal S^{t,x,h}_s)^\top\hat u^h_s\\
&\qquad\qquad+R^{h,1}_s\big]\, \mathrm ds+\sum_{j=1}^{d}\hat Z^{h,j}_s\, \mathrm dW^j_s,\\
&\mathcal R^{t,x,h}_s\hat u^h_s+\mathcal S^{t,x,h}_s\hat X^h_s+B(s)^\top\hat Y^h_s+D(s)^\top\hat Z^h_s+R^{h,2}_s=0,\quad s\in [t,T],\\
&\hat X^h_t = 0, ~~~~~~~~\hat Y^h_T=\mathcal G^{t,x,h}\hat X^h_T+R^{h,3},
\end{aligned}
\right.
\end{equation} 
where the remainder terms are
\begin{equation}\label{pf:prop:cvg:2}
\begin{aligned}
R^{h,1}_s&=(\mathcal Q^{t,x,h}_s-\mathcal Q^{t,x}_s)\nabla_iX^{t,x}_s+(\mathcal S_s^{t,x,h}-\mathcal S_s^{t,x})^\top\nabla_iu^{t,x}_s,\\
R^{h,2}_s&=(\mathcal R^{t,x,h}_s-\mathcal R^{t,x}_s)\nabla_iu^{t,x}_s+(\mathcal S^{t,x,h}_s-\mathcal S^{t,x}_s)\nabla_iX^{t,x}_s,\\
R^{h,3}  &=(\mathcal G^{t,x,h}-\mathcal G^{t,x})\nabla_iX^{t,x}_T.
\end{aligned}
\end{equation}
Applying the estimate \eqref{2:est}, we have, for any $h\in\mathbb R/\{0\}$, 
\begin{equation}
\begin{aligned}
&\mathbb E\bigg[\sup_{t\leq s\leq T}[\vert\hat X^h_s\vert^2+\vert\hat Y^h_s\vert^2]+\int_t^T[\vert\hat Z^h_s\vert^2+\vert\hat u^h_s\vert^2]\, \mathrm ds\bigg]\\
\leq &K\mathbb E\left[\int_t^T[\vert R^{h,1}_s\vert^2+\vert R^{h,2}_s\vert^2]\, \mathrm ds+\vert R^{h,3}\vert^2\right].
\end{aligned}
\end{equation}
By the estimate \eqref{prop:est:2}, $(\bar X^{t,x+he_i},\bar u^{t,x+he_i})$ converges to $(\bar X^{t,x},\bar u^{t,x})$ in $L^2$ sense as $h\to0$. Since $\mathrm D^2_{xx}l,\mathrm D^2_{xu}l,\mathrm D^2_{uu}l,\mathrm D^2_{xx}g$ are continuous in $(x,u)$ and by \eqref{3:lhsg:est}, we have, for $k=1,2,3$,
$$
\lim_{h\to0}R^{h,k}=0,~~~\mbox{in measure }\, \mathrm dt\times \, \mathrm d\mathbb P.
$$
By the boundedness of the second-order derivatives, we deduce \eqref{prop:cvg:1} by applying Lebesgue's dominated convergence theorem.

Finally, to prove \eqref{prop:cvg:2}, for any $(t,x)\in[0,T]\times\mathbb{R}^n$ and $i\in\{1,2,...,n\}$, we set
$$
(\nabla_iX^{t,x}_s,\nabla_iY^{t,x}_s,\nabla_iZ^{t,x}_s,\nabla_iu^{t,x}_s)=(e_i,\nabla_iY^{t,x}_t,0,0),
$$
when $s\leq t$. 
Then, for any $(t,x),(t',x')\in[0,T]\times\mathbb{R}^n$, $(\nabla_i \hat X,\nabla_i \hat Y,\nabla_i \hat Z,\nabla_i \hat u)$ defined by
$$
(\nabla_iX^{t',x'}-\nabla_iX^{t,x},\nabla_iY^{t',x'}-\nabla_iY^{t,x},\nabla_iZ^{t',x'}-\nabla_iZ^{t,x},\nabla_iu^{t',x'}-\nabla_iu^{t,x})
$$
satisfies the following linear Hamiltonian system:
\begin{equation*}
\left\{
\begin{aligned}
&\mathrm d\nabla_i \hat X_s = \big[A(s)\nabla_i \hat X_s+B(s)\nabla_i \hat u_s+\Pi^1_s\big]\, \mathrm ds\\
&\qquad\qquad+\sum_{j=1}^d \big[C_j(s)\nabla_i \hat X_s+D_j(s)\nabla_i \hat u_s+\Pi^{2,j}_s\big]\, \mathrm dW^j_s,\\
&\mathrm d \nabla_i \hat Y_s=-\big[A(s)^\top \nabla_i \hat Y_s +C(s)^\top \nabla_i \hat Z_s+\mathcal Q^{t',x'}_s\nabla_i \hat X_s\\
&\qquad\qquad+(\mathcal S^{t',x'}_s)^\top\nabla_i \hat u_s+\Pi^3_s\big]\, \mathrm ds+\sum_{j=1}^{d}\nabla_i \hat Z_s\, \mathrm dW^j_s,\\
&\mathcal R^{t',x'}_s\nabla_i \hat u_s+\mathcal S^{t',x'}_s\nabla_i \hat X_s+B(s)^\top\nabla_i \hat Y_s+D(s)^\top\nabla_i \hat Z_s+\Pi^4_s=0,\quad s\in [t,T],\\
&\nabla_i \hat X_{t\land t'} = 0, ~~~~~~\nabla_i \hat Y_T=\mathcal G^{t',x'}\nabla_i \hat X_T+\Pi^5,
\end{aligned}
\right.
\end{equation*} 
where $\Pi^5=(\mathcal{G}^{t',x'}-\mathcal{G}^{t,x})\nabla_iX^{t,x}_T$, 
\begin{align*}
&\Pi^1_s=-A(s)e_i\mathbbm 1_{[t,t']}(s),~~\Pi^{2,j}=-C_j(s)e_i\mathbbm 1_{[t,t']}(s),~j=1,...,d,\\
&\Pi^3_s=(\mathcal{Q}^{t',x'}_s-\mathcal{Q}^{t,x}_s)\nabla_iX^{t,x}_s+(\mathcal{S}^{t',x'}_s-\mathcal{S}^{t,x}_s)^\top\nabla_iu^{t,x}_s-\mathcal{Q}^{t',x'}_se_i\mathbbm 1_{[t,t']}(s),\\
&\Pi^4_s=(\mathcal{R}^{t',x'}_s-\mathcal{R}^{t,x}_s)\nabla_iu^{t,x}_s+(\mathcal{S}^{t',x'}_s-\mathcal{S}^{t,x}_s)\nabla_iX^{t,x}_s-\mathcal{S}^{t',x'}_se_i\mathbbm 1_{[t,t']}(s),
\end{align*}
when $t'\geq t$ and
\begin{align*}
&\Pi^1_s=A(s)e_i\mathbbm 1_{[t',t]}(s),~~\Pi^{2,j}=C_j(s)e_i\mathbbm 1_{[t',t]}(s),~j=1,...,d,\\
&\Pi^3_s=(\mathcal{Q}^{t',x'}_s-\mathcal{Q}^{t,x}_s)\nabla_iX^{t,x}_s+(\mathcal{S}^{t',x'}_s-\mathcal{S}^{t,x}_s)^\top\nabla_iu^{t,x}_s+\mathcal{Q}^{t,x}_se_i\mathbbm 1_{[t',t]}(s),\\
&\Pi^4_s=(\mathcal{R}^{t',x'}_s-\mathcal{R}^{t,x}_s)\nabla_iu^{t,x}_s+(\mathcal{S}^{t',x'}_s-\mathcal{S}^{t,x}_s)\nabla_iX^{t,x}_s+\mathcal{S}^{t,x}_se_i\mathbbm 1_{[t',t]}(s),
\end{align*}
when $t'< t$. 
Noting that the form of the above Hamiltonian system is similar to \eqref{3:lhsg}, it satisfies the assumptions of Theorem \ref{th:wp}.
By the estimate \eqref{2:est}, we derive
\begin{equation}
\begin{aligned}
&\mathbb E\bigg[\sup_{0\leq s\leq T}[\vert\nabla_i \hat X_s\vert^2+\vert\nabla_i \hat Y_s\vert^2]+\int_0^T[\vert\nabla_i \hat Z_s\vert^2+\vert\nabla_i \hat u_s\vert^2]\, \mathrm ds\bigg]\\
\leq& K\mathbb E\left[\sum_{k=1}^4\int_{t\land t'}^T\vert \Pi_s^k\vert^2\, \mathrm ds+\vert \Pi^5\vert^2\right].
\end{aligned}
\end{equation}
Since $A,C,\mathcal{Q}^{t,x},\mathcal{Q}^{t',x'},\mathcal{S}^{t,x},\mathcal{S}^{t',x'}$ are uniformly bounded and the rest parts are similar to \eqref{pf:prop:cvg:2}, by \eqref{prop:est:2} and \eqref{3:lhsg:est} again, we deduce, for $k=1,2,3,4,5$,
$$
\lim_{(t',x')\to(t,x)}\Pi^k=0, ~~\mbox{in measure }\, \mathrm dt\times \, \mathrm d\mathbb P.
$$
By applying Lebesgue's dominated convergence theorem again, we conclude \eqref{prop:cvg:2}.
\end{proof}
At the end of this section, we provide the Malliavin derivatives for $(\bar X^{t,x},\bar Y^{t,x},\bar Z^{t,x},\bar u^{t,x})$, prepared for the continuous version of $(\bar Z^{t,x},\bar u^{t,x})$. 
Specifically, the following notations are introduced. 
(One could refer to El Karoui et al. \cite{kpq97} and Nualart \cite{n06} for example.) 
\begin{itemize}
  \item Let $C^{k}_b(\mathbb R^n;\mathbb R^q)$ be the set of all $C^k$ functions from $\mathbb R^n$ to $\mathbb R^q$ all of whose derivatives of order less than and equal to $k$ are bounded.
  \item Let $\mathscr S$ be the set of random variables $\xi$ of the form $\xi=\phi(W(h^1),...,W(h^n))$ where $\phi\in C^{\infty}_b(\mathbb R^n;\mathbb R)$, $W(h^i)=\int^T_0\langle h^i(s),\, \mathrm dW_s\rangle$ and $h^1,...,h^n\in L^{2}(0,T;\mathbb R^d)$.
  \item If $\xi\in\mathscr S$ is of the above form, its Malliavin derivative is defined as the following $\mathbb R^d$ valued process:
$$
\mathbf{D}_{\theta}\xi=\sum_{j=1}^{k} \frac{\partial\phi}{\partial x_j}(W(h^1),...,W(h^k))h^j(\theta), ~~~~0\leq\theta \leq T.
$$
The $i$-th component of $\mathbf{D}_{\theta}\xi$ is denoted by $\mathbf{D}^i_{\theta}\xi$, where $i=1,...,d$. 
For any $\xi\in\mathscr S$, we define its $(1,2)$-norm by:
$$
\Vert\xi\Vert_{\mathbb D_{1,2}}=\left(\mathbb E\left[\vert\xi\vert^2+\int_0^T\vert\mathbf{D}_{\theta}\xi\vert^2\, \mathrm d\theta\right]\right)^{1/2}.
$$
It is known that the operator $\mathbf D$ has a closed extension to the space $\mathbb D_{1,2}$, the closure of $\mathscr S$ with respect to the $(1,2)$-norm $\Vert\cdot\Vert_{\mathbb D_{1,2}}$.
  \item Let $\mathbb L^a_{1,2}(0,T;\mathbb R^n)$ be the set of $\mathbb R^n$ valued $\mathbb F$-adapted processes $\{\phi_t(\omega);(t,\omega)\in[0,T]\times\Omega\}$ satisfying
 \begin{enumerate}[(i)]
  \item For a.e. $t\in[0,T]$, $\phi_t(\cdot)\in(\mathbb D_{1,2})^n$.
  \item $(t,\omega)\mapsto\mathbf D_{\cdot}\phi_t(\omega)\in L^2(0,T;\mathbb R^{n\times d})$ admits an adapted version.
  \item $\Vert\phi\Vert_{\mathbb L^a_{1,2}} :=\left(\mathbb E\left[\int_0^T\vert\phi_s\vert^2\, \mathrm ds+\int_0^T\int_0^T\vert\mathbf D_{\theta}\phi_s\vert^2\, \mathrm d\theta\, \mathrm ds\right]\right)^{1/2} <\infty$.
 \end{enumerate}
\end{itemize}

The following lemma provides the Malliavin derivatives for the solution $(\bar X^{t,x},\bar Y^{t,x},\bar Z^{t,x},\\\bar u^{t,x})$ to the Hamiltonian system \eqref{2:HSo}.
\begin{lemma}\label{lem:md}
Under Assumptions \ref{A3} and \ref{A4}, for any $(t,x)\in[0,T]\times\mathbb R^n$, $(\bar X^{t,x},\bar Y^{t,x},\bar Z^{t,x},\bar u^{t,x})$ belongs to $\mathbb L^a_{1,2}(0,T;\mathbb R^{n+n+nd+m})$. 
Moreover, for any $i=1,...,d$, the following equation admits a unique solution:
\begin{equation}
\left\{
\begin{aligned}
& \mathbf D^i_\theta \bar X^{t,x}_s = C_i(\theta) \bar X^{t,x}_{\theta}+D_i(\theta) \bar u^{t,x}_{\theta}+\sigma_i(\theta)+\int_{\theta}^s\big[A(r)\mathbf D^i_\theta \bar X^{t,x}_r+B(r)\mathbf D^i_\theta \bar u^{t,x}_r\big]\, \mathrm dr\\
&+\sum_{j=1}^d\int_{\theta}^s \big[C_j(r)\mathbf D^i_\theta \bar X^{t,x}_r+D_j(r)\mathbf D^i_\theta \bar u^{t,x}_r\big]\, \mathrm dW^j_r,\\
& \mathbf D^i_\theta \bar Y^{t,x}_s=\mathcal G^{t,x} \mathbf D^i_\theta \bar X^{t,x}_T+\int_s^T\bigg[A(r)^\top \mathbf D^i_\theta \bar Y^{t,x}_r +C(r)^\top \mathbf D^i_\theta \bar Z^{t,x}_r+\mathcal Q_r^{t,x}\mathbf D^i_\theta \bar X^{t,x}_r\\
&+(\mathcal S_r^{t,x})^\top \mathbf D^i_\theta \bar u^{t,x}_r\bigg]\, \mathrm dr-\sum_{j=1}^{d}\int_s^T\mathbf D^i_\theta \bar Z^{t,x,j}_r\, \mathrm dW^j_r,\\
&\mathcal R_s^{t,x}\mathbf D^i_\theta \bar u^{t,x}_s+\mathcal S_s^{t,x}\mathbf D^i_\theta \bar X^{t,x}_s+B(s)^\top \mathbf D^i_\theta\bar  Y^{t,x}_s+D(s)^\top \mathbf D^i_\theta\bar  Z^{t,x}_s=0,
\end{aligned}
\right.
\end{equation} 
which provides a version of $(\mathbf D_{\theta}^i\bar X^{t,x},\mathbf D_{\theta}^i\bar Y^{t,x},\mathbf D^i_{\theta}\bar Z^{t,x},\mathbf D^i_{\theta}\bar u^{t,x})$ when $s\in[\theta,T]$.
When $s\in[0,\theta)$, $(\mathbf D^i_{\theta}\bar X^{t,x}_s,\mathbf D^i_{\theta}\bar Y^{t,x}_s,\mathbf D^i_{\theta}\bar Z^{t,x}_s,\mathbf D^i_{\theta}\bar u^{t,x}_s)=0$.
Moreover, $\{\mathbf D_s\bar  Y^{t,x}_s\}_{s\in[t,T]}$ is a version of $\{\bar Z^{t,x}_s\}_{s\in[t,T]}$.
\end{lemma}
The proof of this lemma is lengthy but similar to Wu et al. \cite{wxy25}, which is put in Appendix \ref{APP-A}.

\section{Properties of value function}\label{sec4}
First of all, to study the properties of the value function via the Hamiltonian system \eqref{2:hst}, we provide a connection between the value function $V$ and the Hamiltonian system \eqref{2:hst}.
Based on the analysis in Section \ref{sec3}, the twice continuous differentiability of $V$ and the representations for $\mathrm D_x V$ and $\mathrm D^2_{xx}V$ are obtained.
For convenience, we denote $\nabla X^{t,x}:=(\nabla_1X^{t,x},\nabla_2X^{t,x},...,\nabla_nX^{t,x})$, similarly for $\nabla Y^{t,x},\nabla Z^{t,x},\nabla u^{t,x}$ (recall the equation \eqref{3:lhsg}). 

\begin{proposition}\label{prop:vfd}
Under Assumptions \ref{A3} and \ref{A4}, the partial derivatives $\mathrm D_xV$ and $\mathrm D^2_{xx}V$ exist and are jointly continuous in $(t,x)$. Moreover, the derivatives admits the following representations:
\begin{equation}
\mathrm D_xV(t,x)=\bar Y^{t,x}_t,~~~\mathrm D^2_{xx}V(t,x)=\nabla Y^{t,x}_t.
\end{equation}
Furthermore, $\mathrm D^2_{xx}V$ is uniformly bounded.
\end{proposition}
\begin{proof}
Firstly, for any $h\in\mathbb R/\{0\}$, we have
\begin{equation}
\begin{aligned}
&\frac 1h[V(t,x+he_i)-V(t,x)]\\
=&\frac 1h[J(t,x+he_i;\bar u^{t,x+he_i})-J(t,x;\bar u^{t,x})]\\
=&\mathbb E\int_0^1\langle\mathrm D_xg(\bar X^{t,x}_T+\lambda h\Delta_i^hX^{t,x}_T),\Delta_i^hX^{t,x}_T\rangle\, \mathrm d\lambda\\
&+\mathbb E\int_t^T\int_0^1\big[\langle\mathrm D_xl(s,\bar X^{t,x}_s+\lambda h\Delta_i^hX^{t,x}_s,\bar u^{t,x}_s+\lambda h\Delta_i^hu^{t,x}_s), \Delta_i^hX^{t,x}_s\rangle\\
&+\langle\mathrm D_ul(s,\bar X^{t,x}_s+\lambda h\Delta_i^hX^{t,x}_s,\bar u^{t,x}_s+\lambda h\Delta_i^hu^{t,x}_s)\mathrm, \Delta_i^hu^{t,x}_s\rangle\big] d\lambda\, \mathrm ds.
\end{aligned}
\end{equation}
By the Lipschitz continuity of $\mathrm D_xg$, H\"older's inequality and Proposition \ref{prop:est}, we have
\begin{align*}
&\bigg\vert\mathbb E\int_0^1\langle\mathrm D_xg(\bar X^{t,x}_T+\lambda h\Delta_i^hX^{t,x}_T), \Delta_i^hX^{t,x}_T\rangle\, \mathrm d\lambda-\mathbb E\big[\langle\mathrm D_xg(\bar X^{t,x}_T),\nabla_iX^{t,x}_T\rangle\big]\bigg\vert\\
\leq&\mathbb E\bigg[\int_0^1\vert \mathrm D_xg(\bar X^{t,x}_T+\lambda h\Delta_i^hX^{t,x}_T)-\mathrm D_xg(\bar X^{t,x}_T)\vert\, \mathrm d\lambda\vert \Delta_i^hX^{t,x}_T\vert\bigg]\\
&+\mathbb E\big[\vert\mathrm D_xg(\bar X^{t,x}_T)\vert\vert\Delta_i^hX^{t,x}_T-\nabla_iX^{t,x}_T\vert\big]\\
\leq&Kh\mathbb E\big[\vert \Delta_i^hX^{t,x}_T\vert^2\big]+\left(\mathbb E\big[\vert\mathrm D_xg(\bar X^{t,x}_T)\vert^2\big]\mathbb E\big[\vert\Delta_i^hX^{t,x}_T-\nabla_iX^{t,x}_T\vert^2\big]\right)^{\frac12}\\
\leq&Kh+K(1+\vert x\vert^2)^{\frac{1}{2}}\left(\mathbb E\big[\vert\Delta_i^hX^{t,x}_T-\nabla_iX^{t,x}_T\vert^2\big]\right)^{\frac12}.
\end{align*}
Similarly, we deduce that
\begin{align*}
&\bigg\vert\mathbb E\int_t^T\int_0^1\big[\langle\mathrm D_xl(s,\bar X^{t,x}_s+\lambda h\Delta_i^hX^{t,x}_s,\bar u^{t,x}_s+\lambda h\Delta_i^hu^{t,x}_s), \Delta_i^hX^{t,x}_s\rangle\\
&+\langle\mathrm D_ul(s,\bar X^{t,x}_s+\lambda h\Delta_i^hX^{t,x}_s,\bar u^{t,x}_s+\lambda h\Delta_i^hu^{t,x}_s), \Delta_i^hu^{t,x}_s\rangle\big]\, \mathrm d\lambda\, \mathrm ds\\
&-\mathbb E\int_t^T\big[\langle\mathrm D_xl(s,\bar X_s^{t,x},\bar u^{t,x}_s),\nabla_iX_s^{t,x}\rangle+\langle\mathrm D_ul(s,\bar X_s^{t,x},\bar u^{t,x}_s),\nabla_iu_s^{t,x}\rangle\big]\, \mathrm ds\bigg\vert\\
\leq&Kh+K(1+\vert x\vert^2)^{\frac{1}{2}}\left(\mathbb E\bigg[\sup_{t \leq s \leq T}\vert\Delta_i^hX^{t,x}_s-\nabla_iX^{t,x}_s\vert^2+\int_t^T\vert\Delta_i^hu^{t,x}_s-\nabla_iu^{t,x}_s\vert^2\, \mathrm ds\bigg]\right)^{\frac12}.
\end{align*}
Letting $h\to0$ and by \eqref{prop:cvg:1}, we derive
\begin{equation}
\begin{aligned}
\frac{\partial V}{\partial x_i}(t,x)=&\lim_{h\to0}\frac 1h[V(t,x+he_i)-V(t,x)]\\
=&\mathbb E\bigg[\langle\mathrm D_xg(\bar X^{t,x}_T),\nabla_iX^{t,x}_T\rangle+\int_t^T[\langle\mathrm D_xl(s,\bar X^{t,x}_s,\bar u^{t,x}_s),\nabla_iX^{t,x}_s\rangle\\
&+\langle\mathrm D_ul(s,\bar X^{t,x}_s,\bar u^{t,x}_s), \nabla_iu^{t,x}_s\rangle]\, \mathrm ds\bigg].
\end{aligned}
\end{equation}
Moreover, applying It\^o's formula to $\langle \bar Y^{t,x}_s,\nabla_iX^{t,x}_s\rangle$ on $[t,T]$ yields
\begin{equation}
\begin{aligned}
&\mathbb E\bigg[\langle\mathrm D_xg(\bar X^{t,x}_T),\nabla_iX^{t,x}_T\rangle\bigg]-\langle \bar Y^{t,x}_t,e_i\rangle\\
=&\mathbb E\bigg[\int_t^T[-\langle\mathrm D_xl(s,\bar X^{t,x}_s,\bar u^{t,x}_s),\nabla_iX^{t,x}_s\rangle+\langle B(s)^\top \bar Y^{t,x}_s+D(s)^\top \bar Z^{t,x}_s ,\nabla_iu^{t,x}_s\rangle]\, \mathrm ds\bigg].
\end{aligned}
\end{equation}
By the stationary condition $\mathcal D[\bar u^{t,x}]=0$ and comparing to $(\partial V)/(\partial x_i)$, we obtain
\begin{equation}
\frac{\partial V}{\partial x_i}(t,x)=\langle \bar Y^{t,x}_t,e_i\rangle,
\end{equation}
which leads to the representation for $\mathrm{D}_xV$. 
Finally, by \eqref{prop:cvg:1}, we have
$$
\lim_{h\to0}\frac 1h[\bar Y^{t,x+he_i}_t-\bar Y^{t,x}_t]=\lim_{h\to0}\Delta_i^h Y^{t,x}_t=\nabla_iY^{t,x}_t.
$$
Hence, $V$ is twice differentiable with respect to $x$. 
The joint continuity of $\mathrm D_xV$ and $\mathrm D^2_{xx}V$ with respect to $(t,x)$ follows from \eqref{prop:est:2} and \eqref{prop:cvg:2}, respectively.
By \eqref{3:lhsg:est}, $\mathrm D^2_{xx}V$ is bounded.
\end{proof}
As a corollary, we are able to represent $\bar Y^{t,x}$ by $\bar X^{t,x}$ and $\mathrm D_x V$ on the whole interval $[t,T]$.
\begin{corollary}\label{cor:dcf}
Under Assumptions \ref{A3} and \ref{A4}, for any initial pair $(t,\xi)\in[0,T]\times L^2_{\mathcal F_t}(\Omega;\mathbb R^n)$, we have
\begin{equation}\label{cor:dcf:1}
\bar Y^{t,\xi}_t(\omega)=\mathrm D_xV(t,\xi(\omega)),~~\mbox{for almost all }\omega\in\Omega.
\end{equation}
Moreover, for any initial pair $(t,x)\in[0,T]\times \mathbb R^n$, we have
\begin{equation}\label{cor:dcf:2}
\mathbb P\Big[\bar Y^{t,x}_s(\omega)=\mathrm D_xV(s,\bar X^{t,x}_s(\omega)), \mbox{ for all } s\in[t,T]\Big]=1.
\end{equation}
\end{corollary}

\begin{proof}
Similarly to the proof of Zhang \cite[Theorem 5.2]{z17}, we first obtain \eqref{cor:dcf:1} by the approximation.
Noting that $\bar Y^{t,x}_s=\bar Y^{s,\bar X_s^{t,x}}_s$, the second equality \eqref{cor:dcf:2} comes from the continuity of $\bar X^{t,x},\bar Y^{t,x}$ and $\mathrm D_xV$.
\end{proof}

\begin{remark}
In the study of deterministic LC optimal control problems, You \cite{yy87, yy97} introduced a class of quasi-Riccati equations, which are first-order PDEs, to characterize $\mathrm D_xV$. However, in the stochastic case considered in this paper, a straightforward formal derivation shows that the corresponding quasi-Riccati equation is a second-order PDE. In other words, it involves the third-order derivatives of $V$. From the viewpoint of \eqref{cor:dcf:2} (or \eqref{1:cnt}), the second-order derivatives of the solutions to the Hamiltonian systems must be taken into account. Consequently, as we mentioned in the Introduction (see the last paragraph on Page \pageref{open}), the unsolved $L^p$ ($p\geq 4$) estimates for solutions of the associated coupled FBSDEs cannot be circumvented. Therefore, the corresponding quasi-Riccati equation is not introduced or studied in this paper.
\end{remark}

In light of the classical LQ theory, it is natural to decouple the linear Hamiltonian system \eqref{3:lhsg} using an SRE.
We introduce the following family of SREs with the parameter $(t,x)\in[0,T]\times \mathbb R^n$:
\begin{equation}\label{4:sre}
\left\{
\begin{aligned}
&\mathrm d\mathcal P^{t,x}_s = -f^{t,x}(s,\mathcal P^{t,x}_s,\Lambda^{t,x}_s)\mathrm{d}s+\sum_{i=1}^{d}\Lambda^{t,x,i}_s\, \mathrm d W_s^i,~~~s\in[t,T],\\
&\mathcal P^{t,x}_T=\mathcal G^{t,x}.
\end{aligned}
\right.
\end{equation}
where the generator $f^{t,x}$ is defined as
\begin{equation*}
\begin{aligned}
&f^{t,x}(s,P,\Lambda):=PA(s)+A(s)^\top P +\sum_{i=1}^{d}\big(C_i(s)^\top P C_i(s)+\Lambda^{i}C_i(s)+C_i(s)^\top \Lambda^{i}\big)+ \mathcal Q^{t,x}_s\\
&-\Big[B(s)^\top P +\sum_{i=1}^{d}\big(D_i(s)^\top P C_i(s)+D_i(s)^\top\Lambda^{i}\big)+\mathcal S^{t,x}_s\Big]^\top\Big[\mathcal R^{t,x}_s+\sum_{i=1}^{d}D(s)^\top PD(s)\Big]^{-1}\\
&\times\Big[B(s)^\top P +\sum_{i=1}^{d}\big(D_i(s)^\top P C_i(s)+D_i(s)^\top\Lambda^{i}\big)+\mathcal S^{t,x}_s\Big],
\end{aligned}
\end{equation*}
and we denote $\Lambda^{t,x}:=(\Lambda^{t,x,1},...,\Lambda^{t,x,d})$ and $\Lambda:=(\Lambda^1,...,\Lambda^d)$.
When $\mathcal{S}^{t,x}\equiv 0$, Tang \cite{t03,t15} proved the well-posedness for the SRE \eqref{4:sre} in the case of $\mathcal{R}_s^{t,x}\geq\delta I_m$.
When $\mathcal{S}^{t,x}\neq 0$, by transforming the coefficients suitably, the SRE \eqref{4:sre} can be rewritten as the same form as the case of $\mathcal{S}^{t,x}\equiv 0$ 
(One could refer to Huang and Yu \cite[Lemma 3.4]{hy14} for example).
Sun et al. \cite{sxy21} proved the well-posedness of the SRE \eqref{4:sre} under the uniform convexity condition \eqref{3:lquc}.
By Sun et al. \cite[Theorem 6.1]{sxy21}, under the uniform convexity condition \eqref{3:lquc}, the matrix valued process $\nabla X^{t,x}_s$ a.s. has an inverse at each $s\in[t,T]$ and the following representation holds:
\begin{equation}\label{4:sre1}
\mathcal P^{t,x}_s=\nabla Y^{t,x}_s(\nabla X^{t,x}_s)^{-1}.
\end{equation}
By Corollary \ref{cor:dcf} and the chain rule, we have
\begin{equation}\label{4:sre2}
\nabla Y^{t,x}_s=\mathrm D^2_{xx}V(s,\bar X^{t,x}_s)\nabla X^{t,x}_s.
\end{equation}
Combining \eqref{4:sre1} and \eqref{4:sre2} leads to
$$
\mathcal P^{t,x}_s=\mathrm D^2_{xx}V(s,\bar X^{t,x}_s).
$$
Specially, letting $s=t$, we have
$$
\mathcal P^{t,x}_t=\mathrm D^2_{xx}V(t,x),
$$
which provides another characterization for $\mathrm D^2_{xx} V$ using SREs. 

Moreover, by Sun et al. \cite[Lemma 6.6]{sxy21}, the following regular condition holds:
\begin{equation}\label{4:lqpd}
  \mathcal R^{t,x}_s+\sum_{i=1}^{d}D(s)^\top\mathcal P^{t,x}_sD(s)=\mathrm D^2_{uu}l(s,\bar X^{t,x}_s,\bar u^{t,x}_s)+\sum_{i=1}^{d}D(s)^\top\mathcal P^{t,x}_sD(s)\geq\delta I_m,~~~~\mbox{a.s. a.e.}.
\end{equation}
The above uniformly positive definiteness is crucial when considering the LQ problem under the uniform convexity condition.
When it comes to the LC problem, we also need a similar regular condition (see \eqref{1:pov}).
However, due to the lack of continuity of $\bar u^{t,x}$, the positive definiteness \eqref{4:lqpd} may fail when $s=t$.
To overcome this difficulty, we introduce a perturbation for $\bar u^{t,x}$ and provide the following proposition.
\begin{proposition}\label{prop:vfpd}
  Under Assumptions \ref{A3} and \ref{A4}, the value function $V$ satisfies the regular condition \eqref{1:pov}.
\end{proposition}
\begin{proof}
Let $(t,x,u)\in[0,T)\times\mathbb R^{n+m}$ be given.
  For any $\epsilon\in(0,T-t]$, set
\begin{equation}\label{prop:vfpd:pf1}
  u^{\epsilon}_s=\left\{
\begin{aligned}
&u,~~~s\in[t,t+\epsilon],\\
&\bar u^{t,x}_s,~~~s\in(t+\epsilon,T].
\end{aligned}
\right.
\end{equation}
Recall the notations \eqref{3:not1}.
We further denote
$$
\begin{pmatrix}
\mathcal Q^\epsilon & (\mathcal S^\epsilon)^\top \\
\mathcal S^\epsilon & \mathcal R^\epsilon
\end{pmatrix}:=\begin{pmatrix}
\mathcal Q^{t,x,u^{\epsilon}}& (\mathcal S^{t,x,u^{\epsilon}})^\top\\
\mathcal S^{t,x,u^{\epsilon}}& \mathcal R^{t,x,u^{\epsilon}}
\end{pmatrix}, ~~~\mathcal G^{\epsilon}:=\mathcal G^{t,x,u^{\epsilon}}.
$$
The following matrix valued Hamiltonian system is introduced:
\begin{equation}
\left\{
\begin{aligned}
&\mathrm d \mathbb X_s^\epsilon= \big[A(s) \mathbb X_s^\epsilon+B(s) \mathbb U_s^\epsilon\big]\, \mathrm ds+\sum_{i=1}^d \big[C_i(s) \mathbb X_s^\epsilon +D_i(s)\mathbb U_s^\epsilon\big]\, \mathrm dW^i_s,\\
&\mathrm d \mathbb Y_s^\epsilon=-\left[A(s)^\top\mathbb Y_s^\epsilon +C(s)^\top\mathbb Z_s^\epsilon+\mathcal Q^{\epsilon}_s\mathbb X_s^\epsilon+(\mathcal S^{\epsilon}_s)^\top\mathbb U_s^\epsilon\right]\, \mathrm ds+\sum_{i=1}^{d} \mathbb Z_s^{\epsilon,i}\, \mathrm dW_s^i,\\
&\mathcal R^{\epsilon}_s\mathbb U_s^\epsilon+\mathcal S^{\epsilon}_s \mathbb X_s^\epsilon+B(s)^\top\mathbb Y_s^\epsilon+D(s)^\top\mathbb Z_s^\epsilon=0,\quad s\in [t,T],\\
&\mathbb X_t^\epsilon =I_n, ~~~~~~~~\mathbb Y_T^\epsilon=\mathcal G^{\epsilon}\mathbb X_T^\epsilon.
\end{aligned}
\right.
\end{equation}
By Lemma \ref{lem:2oc} and Theorem \ref{th:wp} again, the above Hamiltonian system admits a unique solution $(\mathbb X^\epsilon,\mathbb Y^\epsilon,\mathbb Z^\epsilon,\mathbb U^\epsilon)$.
Moreover, applying Sun et al. \cite[Lemma 6.4]{sxy21}, $\mathbb X_s^\epsilon$ is invertible a.s. and 
\begin{equation}\label{prop:vfpd:pf2}
\mathcal P^{\epsilon}_s:=\mathbb Y_s^\epsilon(\mathbb X_s^\epsilon)^{-1}  
\end{equation}
is the solution to the corresponding SRE.
Moreover, we have
\begin{equation}
\mathcal R^{\epsilon}_s+\sum_{i=1}^{d}D(s)^\top\mathcal P^{\epsilon}_sD(s)=\mathrm D^2_{uu}l(s,X^{t,x,u^\epsilon}_s,u^\epsilon_s)+\sum_{i=1}^{d}D(s)^\top\mathcal P^{\epsilon}_sD(s)\geq \delta I_m, \mbox{  a.e. a.s.}.
\end{equation}
By the continuity of $\mathrm D^2_{uu}l,D,X^{t,x,u^\epsilon},\mathcal P^{\epsilon}$ and the definition of $u^\epsilon$, we let $s=t$ and deduce that, for any $\epsilon\in(0,T-t]$,
\begin{equation}\label{prop:vfpd:pf3}
\mathrm D^2_{uu}l(t,x,u)+\sum_{i=1}^{d}D(t)^\top\mathcal P^{\epsilon}_tD(t)\geq \delta I_m.
\end{equation}

Next, we will let $\epsilon\to0$ in \eqref{prop:vfpd:pf3}.
To this aim, we study the limit of $\mathcal P^{\epsilon}_t$ as $\epsilon\to0$.
By letting $s=t$ in \eqref{prop:vfpd:pf2}, we have $\mathcal P^{\epsilon}_t=\mathbb Y_t^\epsilon$.
Noting that
$$
\mathbb E\bigg[\sup_{t\leq s\leq T}\vert X^{t,x,u^{\epsilon}}_s-\bar X^{t,x}_s\vert^2\bigg]\leq K\mathbb E\int_t^{t+\epsilon}\vert u-\bar u^{t,x}_s\vert^2\, \mathrm ds\to0,~~~\mbox{as }\epsilon\to0,
$$
we derive that $(X^{t,x,u^{\epsilon}},u^{\epsilon})$ converges to $(\bar X^{t,x},\bar u^{t,x})$ in $L^2$ sense as $\epsilon \to0$.
Similarly to the proof of \eqref{prop:cvg:1}, we are able to deduce that
\begin{equation*}
\begin{aligned}
\lim_{\epsilon\to0}\mathbb E\bigg[\sup_{t\leq s\leq T}[\vert\mathbb X^{\epsilon}_s-\nabla X^{t,x}_s\vert^2+\vert\mathbb Y^{\epsilon}_s-\nabla Y^{t,x}_s\vert^2]&\\
+\int_t^T(\vert\mathbb Z^{\epsilon}_s-\nabla Z^{t,x}_s\vert^2+\vert\mathbb U^{\epsilon}_s-\nabla u^{t,x}_s\vert^2)\, \mathrm ds\bigg]&=0.
\end{aligned}
\end{equation*}
Specially, we obtain that $\lim_{\epsilon\to0}\mathbb Y^\epsilon_t=\nabla Y_t^{t,x}$, leading to 
$$
\lim_{\epsilon\to0}\mathcal P^{\epsilon}_t=\mathcal P^{t,x}_t=\mathrm D^2_{xx}V(t,x).
$$
By letting $\epsilon\to0$ in \eqref{prop:vfpd:pf3}, we obtain \eqref{1:pov}.
\end{proof}
\begin{remark}\label{rmk:rc}
When further assuming that $g,l$ are convex in $(x,u)$ for any $t\in[0,T]$, the value function $V$ is proven to be convex in $x$ for any $t\in[0,T]$ (see Lemma \ref{lem:covf}).
As a result, $\mathrm D^2_{xx} V(t,x)$ is positive semidefinite for any $(t,x)\in[0,T]\times\mathbb R^n$.
Therefore, when considering Case 1  (see Page \pageref{case}), it is direct to obtain the regular condition \eqref{1:pov} by $\mathrm D^2_{uu}l(t,x,u)\geq\delta I_m$ for any $(t,x,u)\in[0,T]\times\mathbb R^{n+m}$.
\end{remark}
Based on the above properties of the value function $V$, we will present the feedback representation of $\bar u^{t,x}$. 
To this aim, we first make some preparations. 
For any $(t,x,u,p,Q) \in [0,T] \times \mathbb R^{n+m+m} \times \mathbb S^m$, we define
\begin{equation}
L(t,x,u,p,Q) := \langle u, p \rangle +\frac 1 2 \langle Qu, u \rangle +l(t,x,u).
\end{equation}
The following lemma comes from the implicit function theorem.
\begin{lemma}\label{lem:pre}
Let Assumption \ref{A3} hold. 
Then, for any $(t,x,p,Q) \in [0,T] \times \mathbb R^{n+m} \times \mathbb S^m$ satisfying that $Q+\mathrm D^2_{uu}l(t,x,u)$ is positive definite for any $u\in\mathbb R^m$, there exists a unique 
\begin{equation}\label{lem:pre:1}
\bar u := \varphi(t,x,p,Q)
\end{equation}
such that
\begin{equation}\label{lem:pre:2}
L(t,x,\bar u, p, Q) = \inf_{u\in \mathbb R^m} L(t,x, u, p, Q).
\end{equation}

Moreover, let $\delta>0$ be given.
For any $(t,x,p,Q) \in [0,T] \times \mathbb R^{n+m} \times \mathbb S^m$ satisfying the following condition:
\begin{equation}\label{lem:pre:4}
Q+\mathrm D^2_{uu}l(t,x,u)\geq \delta I_m,\mbox{ for any }u\in\mathbb R^m,
\end{equation}
the function $\varphi$ defined above is jointly continuous in $(t,x,p,Q)$ and satisfies the following growth condition:
there exists a constant $K_{\delta}>0$ depending on $\delta$ such that 
\begin{equation}\label{lem:pre:3}
|\varphi(t,x,p,Q)| \leq K_{\delta}(1+|x|+|p|).
\end{equation}

\end{lemma}
\begin{proof}

For any $(t,x,u,p,Q)$, we calculate
\begin{equation}
\mathrm D_u L(t,x,u,p,Q) = p +Qu +\mathrm D_u l(t,x,u)
\end{equation}
and
\begin{equation}\label{lem:pre:pf1}
\mathrm D^2_{uu} L(t,x,u,p,Q) = Q +\mathrm D^2_{uu} l(t,x,u).
\end{equation} 
When $Q+\mathrm D^2_{uu}l(t,x,u)>0$ for any $u\in\mathbb R^m$, we deduce that $\mathrm D^2_{uu} L$ is invertible.
Then, the implicit function theorem (treating $t$ as a parameter and noting that $\mathrm D_u L$ is also continuously differentiable with respect to $x$, $p$ and $Q$)
shows that the equation $\mathrm D_u L(t,x,u,p,Q) =0$ implicitly defines a function, denoted by $\bar u = \varphi(t,x,p,Q)$ (i.e., \eqref{lem:pre:1}), such that
\begin{equation}
\mathrm D_u L(t,x,\bar u,p,Q) =0.
\end{equation} 
Meanwhile, in view of the positive definiteness of $\mathrm D^2_{uu} L$, we obtain \eqref{lem:pre:2}.

Next, we prove the continuity of $\varphi$. 
Let $\delta>0$ be given.
For any given $(t_0,x_0,p_0,Q_0)$ and any $(t,x,p,Q)$ satisfying \eqref{lem:pre:4},
we denote $\bar u_0 = \varphi(t_0,x_0,p_0,Q_0)$ and $\bar u =\varphi(t,x,p,Q)$, respectively. 
Then, we have
\begin{equation}\label{lem:pre:pf2}
\begin{aligned}
0 =\ & \mathrm D_u L(t,x,\bar u,p,Q)\\
=\ & \big[ \mathrm D_u L(t,x,\bar u,p,Q) -\mathrm D_u L(t,x,\bar u_0, p,Q) \big] +\mathrm D_u L(t,x,\bar u_0, p,Q)\\
=\ &\int_0^1\mathrm D^2_{uu} L(t,x,\bar u_0+\gamma (\bar u-\bar u_0), p,Q)(\bar u-\bar u_0)\,\mathrm d\gamma  +\mathrm D_u L(t,x,\bar u_0, p,Q).
\end{aligned}
\end{equation} 
Making the inner product with $\bar u -\bar u_0$ on the equality \eqref{lem:pre:pf2}, we obtain
\[
0 \geq \delta |\bar u -\bar u_0|^2 +\langle \mathrm D_u L(t,x,\bar u_0, p,Q), \bar u -\bar u_0 \rangle.
\]
Then,
\[
\begin{aligned}
|\bar u -\bar u_0| \leq\ & \frac 1 \delta \big|\mathrm D_u L(t,x,\bar u_0, p,Q)\big| = \frac 1 \delta \big| p +Q\bar u_0 +\mathrm D_u l(t,x,\bar u_0) \big|\\
\rightarrow\ & \frac 1 \delta \big| p_0 +Q_0\bar u_0 +\mathrm D_u l(t_0,x_0,\bar u_0) \big| = \frac 1 \delta \big|\mathrm D_u L(t_0,x_0,\bar u_0, p_0,Q_0)\big| =0,
\end{aligned}
\]
as $(t,x,p,Q) \rightarrow (t_0,x_0,p_0,Q_0)$. 
Due to the arbitrariness of $(t_0,x_0,p_0,Q_0)$, the continuity of $\varphi$ is proven.

Finally, we prove the growth condition \eqref{lem:pre:3}.
For any $(t,x,p,Q) \in [0,T] \times \mathbb R^{n+m} \times \mathbb S^m$ satisfying \eqref{lem:pre:4}, we deduce that $L$ is uniformly convex in $u$.
Therefore, we have
\begin{equation}
\begin{aligned}
\delta |\bar u|^2 \leq\ & \langle \mathrm D_u L(t,x,\bar u,p,Q) -\mathrm D_uL(t,x,0,p,Q), \bar u \rangle = -\langle \mathrm D_uL(t,x,0,p,Q), \bar u \rangle\\
=\ & -\langle p +\mathrm D_u l(t,x,0), \bar u \rangle \leq |p +\mathrm D_u l(t,x,0)| |\bar u|,
\end{aligned}
\end{equation}
leading to \eqref{lem:pre:3}.
\end{proof}

\begin{remark}
In the proof of Lemma \ref{lem:pre}, the application of the implicit function theorem can further yield that $\varphi$ is continuously differentiable with respect to $(x,p,Q)$.
We denote the derivatives with respect to $p$ and $Q$ as $\mathrm D_p \varphi$ and $\mathrm D_Q \varphi$, respectively. 
Then, for any given $\delta>0$, there exists a constant $K_{\delta}$ depending on $\delta$ such that, for any $(t,x,p,Q) \in [0,T] \times \mathbb R^{n+m} \times \mathbb S^m$ satisfying \eqref{lem:pre:4},
$\mathrm D_x\varphi$ and $\mathrm D_p\varphi$ are uniformly bounded by $K_{\delta}$, and $\mathrm D_Q\varphi$ is bounded by $K_{\delta}(1+|x| +|p|)$.
\end{remark}

Combining Lemmas \ref{lem:md} and \ref{lem:pre}, we conclude the following proposition, providing the continuous version of $(\bar Z^{t,x},\bar u^{t,x})$.
\begin{proposition}\label{prop:cv}
Under Assumptions \ref{A3} and \ref{A4}, for any $(t,x)\in[0,T]\times\mathbb R^n$, the trajectories of the processes $\bar u^{t,x}$ and $\bar Z^{t,x}$ are continuous in $s$ a.s.
Moreover, the optimal control $\bar u^{t,x}$ admits the following state feedback form:
\begin{equation}\label{prop:cv:1}
\bar u^{t,x}_s=\bar{\mathbbm u}(s,\bar X^{t,x}_s),
\end{equation} 
where
\begin{equation}\label{prop:cv:2}
\begin{aligned}
\bar{\mathbbm u}(t,x) :=& \varphi \bigg( t,x, B(t)^\top \mathrm D_u V(t,x) +\sum_{i=1}^d D_i(t)^\top \mathrm D^2_{xx} V(t,x) \big[ C_i(t)x +\sigma_i(t) \big],\\ 
&\sum_{i=1}^d D_i(t)^\top \mathrm D^2_{xx} V(t,x) D_i(t) \bigg).
\end{aligned}
\end{equation}
Consequently, $\bar{\mathbbm u}$ is jointly continuous in $(t,x)$ and linear growth in $x$.
\end{proposition}
\begin{proof}
By Corollary \ref{cor:dcf} and Lemma \ref{lem:md}, we have 
\begin{equation}\label{prop:cv:pf1}
\begin{aligned}
\bar Z^{t,x,i}_s&=\mathbf D_s^i\bar Y^{t,x}_s =\mathrm D^2_{xx}V(s,\bar X^{t,x}_s)\mathbf D_s^i\bar X^{t,x}_s\\
&=\mathrm D^2_{xx}V(s,\bar X^{t,x}_s)[ C_i(s)\bar X^{t,x}_s+D_i(s)\bar u^{t,x}_s+\sigma_i(s)].
\end{aligned}
\end{equation} 
Substituting it into the stationary condition leads to
\begin{equation*}
\begin{aligned}
&\mathrm D_ul(s,\bar X^{t,x}_s,\bar u^{t,x}_s)+B(s)^\top \mathrm D_{x}V(s,\bar X^{t,x}_s)\\
&+\sum_{i=1}^dD_i(s)^\top\mathrm D^2_{xx}V(s,\bar X^{t,x}_s)[ C_i(s) \bar X^{t,x}_s+D_i(s)\bar u^{t,x}_s+\sigma_i(s)]=0.
\end{aligned}
\end{equation*} 
By Proposition \ref{prop:vfpd}, there exists a constant $\delta >0$ such that 
\begin{equation*}
\mathrm D^2_{uu}l(s,\bar X^{t,x}_s,\bar u^{t,x}_s)+\sum_{i=1}^dD_i(s)^\top\mathrm D^2_{xx}V(s,\bar X^{t,x}_s)D_i(s)\geq\delta I_m.
\end{equation*} 
Applying Lemma \ref{lem:pre}, we derive the state feedback form \eqref{prop:cv:1} for the optimal control $\bar u^{t,x}$. 
Due to the continuity of $\varphi$, we deduce that $\bar{\mathbbm u} $ is jointly continuous in $(t,x)$, which also leads to the continuity of $\bar u^{t,x}$.
Since $\mathrm D^2_{xx}V$ is bounded (see Proposition \ref{prop:vfd}), $\bar{\mathbbm u} $ is linear growth in $x$. 
Moreover, by \eqref{prop:cv:pf1}, $\bar Z^{t,x}$ is also verified to be continuous in $s$.
\end{proof}
By Lemma \ref{lem:pre} and the definitions of $\bar{\mathbbm u}, L,H$ and $\mathcal H$, we deduce that, for any $(t,x) \in [0,T] \times \mathbb R^n$, 
\begin{equation}\label{4:inf}
H \big( t,x,\mathrm D_x V(t,x), \mathrm D^2_{xx} V(t,x) \big)=\mathcal H \big( t,x,\mathrm D_x V(t,x), \mathrm D^2_{xx} V(t,x), \bar{\mathbbm u}(t,x)\big).
\end{equation}

\section{Dynamic programming and classical solution to HJB equation}\label{sec5}
It is well-known that the value function $V$ of the optimal control problem is the unique viscosity solution to the corresponding HJB equation under some appropriate conditions. 
In this section, we would like to prove that $V$ is the unique classical solution to the HJB equation \eqref{1:HJB}, which needs to prove $V\in C^{1,2}([0,T]\times\mathbb R^n;\mathbb R)$. 

We have already proven that $V$ is twice continuously differentiable with respect to $x$ in Proposition \ref{prop:vfd} and its second-order derivative $\mathrm D^2_{xx}V$ is uniformly bounded.
Moreover, Propositions \ref{prop:vfpd} shows that the value function also satisfies the regular condition \eqref{1:pov}.
As mentioned in Remark \ref{rmk:rc}, under Case 1 (see Page \pageref{case}), the regular condition \eqref{1:pov} can be replaced by the condition that $V$ is convex in $x$ for any $t\in[0,T]$.
When it reduces to the LQ case, the value function $V$ is quadratic in $x$ and the regular condition \eqref{1:pov} is reduced to the following uniformly positive definiteness:
\begin{equation}\label{5:lqrc}
R(t)+\sum_{i=1}^dD_i(t)^\top P(t) D_i(t)\geq \delta I_m,~~~~\mbox{for any }t\in[0,T],
\end{equation}
where $P$ is the solution to the corresponding Riccati equation.
The differentiability with respect to $t$ will be shown in this section by the dynamic programming approach.

First of all, we provide the following lemma, which is prepared for the dynamic programming principle.
\begin{lemma}\label{lem:dpp}
  Under Assumptions \ref{A3} and \ref{A4}, for any $(t,\xi)\in[0,T]\times L^2_{\mathcal F_t}(\Omega;\mathbb R^n)$,
  \begin{equation}
    V(t,\xi)=\mathbb E\left[ g(\bar X^{t,\xi}_{T})+\int_t^{T}l(s,\bar X^{t,\xi}_s,\bar u^{t,\xi}_s)\, \mathrm d s\bigg\vert\mathcal F_t\right].
  \end{equation}
\end{lemma}
The proof is direct but lengthy, so we would like to put it in Appendix \ref{APP-B}. 

With the help of Lemma \ref{lem:dpp}, we now prove the following dynamic programming principle.
\begin{theorem}\label{th:dpp}
  Under Assumptions \ref{A3} and \ref{A4}, for any $0\leq t< t+h< T$, the following equality holds:
  \begin{equation}\label{th:dpp1}
    V(t,x)=\mathbb E\left[ V(t+h,\bar X^{t,x}_{t+h})+\int_t^{t+h}l(s,\bar X^{t,x}_s,\bar u^{t,x}_s)\, \mathrm d s\right].
  \end{equation} 
\end{theorem}

\begin{proof}
Noting that 
$$
(\bar X^{t+h,\bar X^{t,x}_{t+h}}_s,\bar Y^{t+h,\bar X^{t,x}_{t+h}}_s,\bar Z^{t+h,\bar X^{t,x}_{t+h}}_s,\bar u^{t+h,\bar X^{t,x}_{t+h}}_s)=(\bar X^{t,x}_s,\bar Y^{t,x}_s,\bar Z^{t,x}_s,\bar u^{t,x}_s),
$$
for almost all $(s,\omega)\in[t+h,T]\times\Omega$ and by Lemma \ref{lem:dpp}, we derive
\begin{equation}
  V(t+h,\bar X^{t,x}_{t+h})=\mathbb E\left[ g(\bar X^{t,x}_{T})+\int_{t+h}^{T}l(s,\bar X^{t,x}_s,\bar u^{t,x}_s)\, \mathrm d s\bigg\vert\mathcal F_t\right].
\end{equation}
Comparing with $V(t,x)=J(t,x;\bar u^{t,x})$, we conclude \eqref{th:dpp1} by the tower property for conditional expectations.
\end{proof}

Applying Theorem \ref{th:dpp}, we now prove the following existence result of the classical solution of the HJB equation \eqref{1:HJB}.

\begin{theorem}\label{th:HJBe}
Under Assumptions \ref{A3} and \ref{A4}, the value function $V\in C^{1,2}([0,T]\times \mathbb R^n;\mathbb R)$. 
Moreover, $V$ is a classical solution to the HJB equation \eqref{1:HJB}.
\end{theorem}

\begin{proof}
For any given $(t,x)\in[0,T]\times \mathbb R^n$ and $h>0$, by It\^o's formula and Theorem \ref{th:dpp}, we have
\begin{equation}\label{th:HJBe:pf1}
\begin{aligned}
\frac{1}{h}[V(t+h,x)-V(t,x)]=&-\frac{1}{h}\mathbb E[V(t+h,\bar X^{t,x}_{t+h})-V(t+h,x)]\\
&+\frac 1h\mathbb E[V(t+h,\bar X^{t,x}_{t+h})-V(t,x)]\\
=&-\frac{1}{h}\int_t^{t+h}\mathbb E[F^{t,x,h}(s)]\, \mathrm d s,
\end{aligned}
\end{equation}
where
\begin{equation}
\begin{aligned}
F^{t,x,h}(s)=&l(s,\bar X^{t,x}_s,\bar u^{t,x}_s)+\langle \mathrm D_xV(t+h,\bar X^{t,x}_s),A(s)\bar X^{t,x}_s+B(s)\bar u^{t,x}_s+b(s)\rangle\\
&+\frac12\sum_{i=1}^d\langle \mathrm D^2_{xx}V(t+h,\bar X^{t,x}_s)[C_i(s)\bar X^{t,x}_s+D_i(s)\bar u^{t,x}_s+\sigma_i(s)],\\
&\qquad C_i(s)\bar X^{t,x}_s+D_i(s)\bar u^{t,x}_s+\sigma_i(s)\rangle.
\end{aligned}
\end{equation}
By Proposition \ref{prop:cv} and the continuity of $l,\bar{\mathbbm u},\mathrm D_xV$ and $\mathrm D^2_{xx}V$, we derive the following convergence:
$$
\lim_{\substack{(h,s)\to(0,t)\\ t\leq s\leq t+h}}F^{t,x,h}(s)=F^{t,x,0}(t),
$$
where
\begin{equation}
\begin{aligned}
&F^{t,x,0}(t)=l(t,x,\bar{\mathbbm u}(t,x))+\langle \mathrm D_xV(t,x),A(t)x+B(t)\bar{\mathbbm u}(t,x)+b(t)\rangle\\
&+\frac12\sum_{i=1}^d\langle \mathrm D^2_{xx}V(t,x)[C_i(t)x+D_i(t)\bar{\mathbbm u}(t,x)+\sigma_i(t)],C_i(t)x+D_i(t)\bar{\mathbbm u}(t,x)+\sigma_i(t)\rangle.
\end{aligned}
\end{equation}
By \eqref{4:inf}, we rewrite $F^{t,x,0}(t)$ as follows:
\begin{equation}\label{th:HJBe:pf2}
F^{t,x,0}(t)=(\mathcal LV)(t,x)+H(t,x,\mathrm D_xV(t,x),\mathrm D^2_{xx}V(t,x)).
\end{equation}
By the boundedness of $\mathrm D^2_{xx}V$, the linear growth of $\bar{\mathbbm u}(t,x)$ and $\mathrm D_xV$, there is a constant $K>0$ which is independent of $t$ and $h$ such that
\begin{equation}\label{th:HJBe:pf3}
\mathbb E[\vert F^{t,x,h}(s)\vert]\leq K\mathbb E[1+\vert \bar X^{t,x}_s\vert^2+\vert \bar u^{t,x}_s\vert^2]\leq K(1+\vert x\vert^2).
\end{equation}
By Lebesgue's dominated convergence theorem, 
$$
\lim_{\substack{(h,s)\to(0,t)\\ t\leq s\leq t+h}}\mathbb E[F^{t,x,h}(s)]=F^{t,x,0}(t).
$$ 
Then, by Lebesgue's differential theorem, we obtain that $V$ is right differentiable with respect to $t$ by letting $h\to0$ in \eqref{th:HJBe:pf1}. 

For the left differentiability, we consider 
\begin{equation}
\frac{1}{h}[V(t,x)-V(t-h,x)]=-\frac{1}{h}\int_{t-h}^{t}\mathbb E[F^{t-h,x,h}(s)]\, \mathrm d s.
\end{equation}
Similarly, noting that \eqref{th:HJBe:pf3} still holds, we only need to prove 
$$
\lim_{\substack{(h,s)\to(0,t)\\t-h\leq s\leq t}}F^{t-h,x,h}(s)=F^{t,x,0}(t).
$$ 
By the estimate \eqref{prop:est:2}, we obtain (recall that $\bar X^{t,x}_s=x$ when $0\leq s\leq t$)
\begin{equation}
\mathbb E\left[\sup_{t-h\leq s\leq t}\vert \bar X^{t-h,x}_s-x\vert^2\right]\leq\mathbb E\left[\sup_{0\leq s\leq T}\vert \bar X^{t-h,x}_s-\bar X^{t,x}_s\vert^2\right]\leq K(1+\vert x\vert^2)h,
\end{equation}
i.e., $\bar X^{t-h,x}_s\to x$ as $h\to0$ in $L^2$ sense. 
Then, according to the continuity of $l,\bar{\mathbbm u},\mathrm D_xV$ and $\mathrm D^2_{xx}V$, we conclude that $F^{t-h,x,h}(s)$ converges to $F^{t,x,0}(t)$ in probability, which shows the left differentiability with respect to $t$ of $V$. 
By \eqref{th:HJBe:pf2}, $V$ satisfies the HJB equation \eqref{1:HJB}.
\end{proof}
Now, we turn to the uniqueness of the classical solution to the HJB equation \eqref{1:HJB}.
The stochastic verification theorem is a powerful tool to prove the uniqueness, which involves the optimal feedback controls.
However, by Lemma \ref{lem:pre} and Proposition \ref{prop:cv}, we are not able to obtain that $\bar{\mathbbm u}$ is Lipschitz continuous (even locally Lipschitz continuous) in $x$ due to the existence of $\mathrm D^2_{xx}V$.
Thus, when employing the feedback control $\bar{\mathbbm u}$, the controlled SDE \eqref{1:SDEt} may be ill-posed in the strong formulation, which means that the filtered probability space $(\Omega, \mathcal F, \mathbb F, \mathbb P)$ and the Brownian motion $W$ are fixed.

To overcome this difficulty, we would like to consider the weak formulation of the optimal control problem, which is defined as follows
(one could refer to Yong and Zhou \cite[Chapter 2, Definition 4.2 and Chapter 5, Definition 6.1]{yz99} for more details):

\begin{definition}
  A 6-tuple $\pi=(\tilde \Omega,\tilde{\mathcal F},\tilde{\mathbb F},\tilde{\mathbb P},\tilde W,\tilde u)$ is called a weak admissible control, and $(\tilde X^{t,x,\tilde u},\tilde u)$ a weak admissible pair, if
\begin{enumerate}[(i)]
  \item $(\tilde \Omega,\tilde{\mathcal F},\tilde{\mathbb F},\tilde{\mathbb P})$ is a complete filtered probability space satisfying the usual condition, where $\tilde{\mathbb F}=(\tilde{\mathcal F}_t)_{0\leq t\leq T}$.
  \item $\tilde W$ is a $d$-dimensional $\tilde{\mathbb F}$-Brownian motion defined on $(\tilde \Omega,\tilde{\mathcal F},\tilde{\mathbb F},\tilde{\mathbb P})$. 
  \item $u$ is an $\tilde{\mathbb F}$-adapted process taking values in $\mathbb R^m$ such that $\tilde{\mathbb E}\int_t^T\vert u_s\vert^2\mathrm ds<\infty$, where $\tilde{\mathbb E}$ is the expectation defined on the probability space $(\tilde \Omega,\tilde{\mathcal F},\tilde{\mathbb F},\tilde{\mathbb P})$.
  \item $\tilde X^{t,x,\tilde u}$ is an $\tilde{\mathbb F}$-adapted process satisfying the SDE \eqref{1:SDEt} defined on $(\tilde \Omega,\tilde{\mathcal F},\tilde{\mathbb F},\tilde{\mathbb P})$ with Brownian motion $\tilde W$ and control $\tilde u$.
  \item The cost functional $\tilde J(t,x;\pi)$ defined as
  \begin{equation}
  \tilde J(t,x;\pi)=\tilde{\mathbb E}\left[g(\tilde X^{t,x,\tilde u}_T)+\int_t^T l(s,\tilde X^{t,x,\tilde u}_s,\tilde u_s)\, \mathrm ds\right],
  \end{equation}
  exists and is finite, i.e., $|\tilde J(t,x;\pi)|<\infty$. 
\end{enumerate}
\end{definition}
All of the weak admissible controls form the weak admissible control set $\mathscr U^{\mathrm w}_{\mathrm{ad}}[t,T]$.
Based on the above definition, we further define the admissible and optimal feedback control in the weak formulation.
\begin{definition}\label{def:fbc}
  A measurable function $\mathbbm u:[0,T]\times \mathbb R^n\to \mathbb R^m$ is called a weak admissible feedback control, 
  if for any $(t,x)\in[0,T)\times\mathbb R^n$, there is a weak solution $(\tilde \Omega,\tilde{\mathcal F},\tilde{\mathbb F},\tilde{\mathbb P},\tilde W,\tilde X^{t,x,\mathbbm u})$ to the following closed-loop system:
 \begin{equation}\label{def:fbc:1}
\left\{
\begin{aligned}
\mathrm d\tilde X_s^{t,x,\mathbbm u}=& \big[A(s)\tilde X_s^{t,x,\mathbbm u}+B(s)\mathbbm u(s,\tilde X_s^{t,x,\mathbbm u})+b(s)\big]\, \mathrm ds \\
&+\sum_{i=1}^d \big[C_i(s)\tilde X_s^{t,x,\mathbbm u} +D_i(s)\mathbbm u(s,\tilde X_s^{t,x,\mathbbm u})+\sigma_i(s)\big]\, \mathrm dW^i_s,\quad s\in [t,T],\\
 \tilde X_t^{t,x,\mathbbm u} =&x,
\end{aligned}
\right.
\end{equation} 
such that $(\tilde \Omega,\tilde{\mathcal F},\tilde{\mathbb F},\tilde{\mathbb P},\tilde W,\mathbbm u(\cdot,\tilde{X}^{t,x}_{\cdot}))\in\mathscr U_{\mathrm{ad}}^{\mathrm w}[t,T]$.\\
A weak admissible feedback control $\bar{\mathbbm u}$ is called optimal, if for any $(t,x)\in[0,T)\times\mathbb R^n$, 
the weak admissible control $\bar {\pi}=(\tilde \Omega,\tilde{\mathcal F},\tilde{\mathbb F},\tilde{\mathbb P},\tilde W,\bar{\mathbbm u}(\cdot,\tilde{X}^{t,x,\bar{\mathbbm u}}_{\cdot}))$ is an optimal one in the weak formulation, i.e.,
  \begin{equation}
  \tilde J(t,x;\bar{\pi})=\inf_{\pi\in\mathscr U_{\mathrm {ad}}^{\mathrm w}[t,T]}\tilde J(t,x;\pi).
  \end{equation}
\end{definition}

Using the weak formulation, we now derive the following stochastic verification theorem. 
\begin{theorem}\label{th:HJBu}
  Let Assumptions \ref{A3} and \ref{A4} hold. 
  Suppose $V$ is a classical solution of the HJB equation \eqref{1:HJB} satisfying the regular condition \eqref{1:pov} and that $\mathrm D^2_{xx}V$ is uniformly bounded.
  Then $\bar{\mathbbm u}$ defined as \eqref{prop:cv:2} is an optimal feedback control for Problem (LC)$_{(t,x)}$.
  Moreover,
\begin{equation}\label{th:HJBu:1}
  V(t,x)=\inf_{\pi\in\mathscr U_{\mathrm {ad}}^{\mathrm w}[t,T]} \tilde J(t,x;\pi).
\end{equation}
\end{theorem}
\begin{proof}
  By Lemma \ref{lem:pre}, we have $\bar{\mathbbm{u}}$ is continuous in $x$ and linear growth in $x$. Then, by the SDE theory (see Pardoux and R\u a\c scanu \cite[Chapter 3, Theorem 3.54]{pr14} for example), the SDE \eqref{def:fbc:1} has a weak solution. Since $\bar{\mathbbm{u}}$ is linear growth in $x$, $\bar{\pi}$ is a weak admissible control.
  
  Now, for any weak admissible control $\pi\in\mathscr U_{\mathrm{ad}}^{\mathrm w}[t,T]$, applying It\^o's formula to $V(s,\tilde{X}^{t,x,\tilde{u}}_s)$ on $[t,T]$, we obtain
\begin{equation}\label{th:HJBu:pf1}
\begin{aligned}
  &\tilde{\mathbb E}[g(\tilde{X}^{t,x,\tilde{u}}_T)]-V(t,x)\\
  =&\tilde{\mathbb E}\int_t^T[\frac{\partial V}{\partial t}(s,\tilde{X}^{t,x,\tilde{u}}_s)+\langle\mathrm D_xV(s,\tilde{X}^{t,x,\tilde{u}}_s),A(s)\tilde{X}^{t,x,\tilde{u}}_s+B(s)\tilde u_s+b(s)\rangle\\
  &+\frac{1}{2}\sum_{i=1}^d \langle \mathrm D^2_{xx}V(s,\tilde{X}^{t,x,\tilde{u}}_s)[C_i(s)\tilde{X}^{t,x,\tilde{u}}_s+D_i(s)\tilde u_s+\sigma_i(s)],\\
  &\qquad C_i(s)\tilde{X}^{t,x,\tilde{u}}_s+D_i(s)\tilde u_s+\sigma_i(s)\rangle]\, \mathrm d s.
\end{aligned}
\end{equation} 
Since $V$ is a classical solution to the HJB equation \eqref{1:HJB}, the above equality can be rewritten as follows when compared with $\tilde J(t,x;\pi)$:
\begin{equation}\label{th:HJBu:pf2}
\begin{aligned}
  &\tilde J(t,x;\pi)-V(t,x)\\
 =&\tilde{\mathbb E}\int_t^T[\mathcal{H}(s,\tilde{X}^{t,x,\tilde{u}}_s,\mathrm D_{x}V(s,\tilde{X}^{t,x,\tilde{u}}_s),\mathrm D^2_{xx}V(s,\tilde{X}^{t,x,\tilde{u}}_s),\tilde{u}_s)\\
   &-{H}(s,\tilde{X}^{t,x,\tilde{u}}_s,\mathrm D_{x}V(s,\tilde{X}^{t,x,\tilde{u}}_s),\mathrm D^2_{xx}V(s,\tilde{X}^{t,x,\tilde{u}}_s))]\, \mathrm d s.
\end{aligned}
\end{equation}
By the definitions of $H$ and $\mathcal H$, we deduce $\tilde J(t,x;\pi)\geq V(t,x)$ for any weak admissible control $\pi$. 
Moreover, by the definition of $\bar{\mathbbm u}$, \eqref{th:HJBu:pf2} demonstrates $ \tilde J(t,x;\bar{\pi})=V(t,x)$. 
As a result, $\bar{\mathbbm u}$ is an optimal feedback control and \eqref{th:HJBu:1} holds.
\end{proof}

As a corollary of Theorem \ref{th:HJBu}, the uniqueness result of the HJB equation \eqref{1:HJB} is presented below.

\begin{corollary}\label{cor:HJBu}
  Under Assumptions \ref{A3} and \ref{A4}, the HJB equation \eqref{1:HJB} admits at most one classical solution satisfying the regular condition \eqref{1:pov} and that $\mathrm D^2_{xx}V$ is uniformly bounded.
Particularly, if replacing Assumption \ref{A4} by Case 1 (see Page \pageref{case}), the HJB equation \eqref{1:HJB} admits at most one classical solution which is convex in $x$ for any $t\in[0,T]$ and that $\mathrm D^2_{xx}V$ is uniformly bounded.
\end{corollary}

\begin{proof}
By Theorem \ref{th:HJBu}, all of the classical solutions to the HJB equation \eqref{1:HJB} is represented by \eqref{th:HJBu:1}, showing the uniqueness.
In Case 1, the regular condition \eqref{1:pov} is satisfied when the classical solution is convex in $x$ for any $t\in[0,T]$ (see Remark \ref{rmk:rc}).
The proof is completed.
\end{proof}
It is clear that the uniqueness of the viscosity solution to the HJB equation \eqref{1:HJB} implies the uniqueness of the classical solution.
To the best of our knowledge, in the literature, the uniqueness results for viscosity solutions of HJB equations arising from stochastic optimal control and stochastic differential game problems usually need one or more of the following assumptions:
(i) The functions $g$ and $l$ in the cost functional are linear growth in $x$.
(ii) The control domain $U\in \mathbb R^m$ is compact.
(iii) The function $l$ is bounded in $u$.
See Peng \cite{ypfw97}, Yong and Zhou \cite{yz99}, Buckdahn and Li \cite{bl08} and Wu and Yu \cite{wy08} for example.
However, the assumptions of Corollary \ref{cor:HJBu} do not require any of the above ones.

Combining Theorem \ref{th:HJBe} and Corollary \ref{cor:HJBu}, we conclude that the value function $V$ is the unique classical solution to the HJB equation \eqref{1:HJB} satisfying the regular condition \eqref{1:pov} and that $\mathrm D^2_{xx}V$ is uniformly bounded.
Moreover, by the uniqueness result, we derive that, for any $(t,x)\in[0,T]\times \mathbb R^n$,
\begin{equation}
\inf_{u\in\mathscr U_{\mathrm {ad}}[t,T]} J(t,x;u)=  V(t,x)=\inf_{\pi\in\mathscr U_{\mathrm {ad}}^{\mathrm w}[t,T]} \tilde J(t,x;\pi).
\end{equation}
In other words, the value functions in the strong formulation and the weak formulation are the same.

At the end of this section, we prove the existence and uniqueness result of the closed-loop system (see \eqref{def:fbc:1} in the weak formulation), 
which illustrates that the weak optimal feedback control $\bar {\mathbbm u}$ is also optimal in the strong formulation.
\begin{theorem}
  Under Assumptions \ref{A3} and \ref{A4}, for any initial pair $(t,x)$, the following closed-loop system admits a unique strong solution $\bar X^{t,x}\in L^2_{\mathbb F}(\Omega;C(t,T;\mathbb R^n))$:
   \begin{equation}\label{th:cls:1}
\left\{
\begin{aligned}
\mathrm d\bar X^{t,x}_s=& \big[A(s)\bar X^{t,x}_s+B(s)\bar{\mathbbm u}(s,\bar X^{t,x}_s)+b(s)\big] \, \mathrm ds \\
&+\sum_{i=1}^d \big[C_i(s)\bar X^{t,x}_s +D_i(s)\bar{\mathbbm u}(s,\bar X^{t,x}_s)+\sigma_i(s)\big]\, \mathrm dW^i_s,\quad s\in [t,T],\\
 \bar X^{t,x}_t =&x.
\end{aligned}
\right.
\end{equation}
Moreover, $\bar u^{t,x}=\bar{\mathbbm u}(\cdot,\bar X^{t,x}_{\cdot})$ is the unique optimal control for Problem (LC)$_{(t,x)}$ in the strong formulation.
Consequently, $\bar{\mathbbm u}$ is the optimal feedback control in the strong formulation.
\end{theorem}
\begin{proof}
  Firstly, by Proposition \ref{prop:cv}, the solution $\bar X^{t,x}$ to the Hamiltonian system \eqref{2:hslq} with initial pair $(t,x)$ satisfies the closed-loop system \eqref{th:cls:1}, leading to the existence.

  For the uniqueness, let $\mathring X^{t,x}\in L^2_{\mathbb F}(\Omega;C(t,T;\mathbb R^n))$ be another solution to the closed-loop system \eqref{th:cls:1}.
  Denote $\mathring{u}^{t,x}_s=\bar {\mathbbm u}(s,\mathring X^{t,x}_{s})$, for any $s\in[t,T]$.
  By the linear growth of $\bar {\mathbbm u}$, we have
  \[
  \mathbb E\int_t^T\vert\mathring{u}^{t,x}_s\vert^2 \, \mathrm ds <\infty.
  \]
  Similarly to \eqref{th:HJBu:pf1} and \eqref{th:HJBu:pf2}, we apply It\^o's formula to $V(s,\mathring X^{t,x}_s)$ on $[t,T]$ and compare with $J(t,x;\mathring u^{t,x})$, leading to the following equality:
  \[
  V(t,x)=J(t,x;\mathring{u}^{t,x}).
  \]
  Therefore, $\mathring{u}^{t,x}$ is the optimal control for Problem (LC)$_{(t,x)}$.
  By the uniqueness of the optimal control, we obtain $\mathring{u}^{t,x}=\bar u^{t,x}$ in the space $\mathscr U_{\mathrm {ad}}[t,T]$.
  Applying the standard estimate for SDEs, we derive that
  $$
\mathbb E \left[\sup_{t\leq s\leq T}\vert \mathring X^{t,x}_s-\bar X^{t,x}_s\vert^2\right]\leq K\mathbb E\int_t^T\vert\mathring u^{t,x}_s-\bar u^{t,x}_s\vert^2 \,\mathrm ds=0.
  $$
Consequently, $\mathring X^{t,x}=\bar X^{t,x}$ in the space $L^2_{\mathbb F}(\Omega;C(t,T;\mathbb R^n))$.
  The proof is completed.
\end{proof}

\begin{appendix}
  \section{Proof of Lemma \ref{lem:md}}\label{APP-A}
Recall the contractive mapping $\mathcal T$ defined in the proof of Theorem \ref{th:wp}. 
We would like to prove that $\mathcal T$ maps $\mathbb L^a_{1,2}(0,T;\mathbb R^m)$ into itself and the corresponding iterative sequence also converges under the norm $\Vert\cdot\Vert_{\mathbb L^a_{1,2}}$.

Firstly, for any $u\in\mathbb L^a_{1,2}(0,T;\mathbb R^m)$, we have $(X^{t,x,u},Y^{t,x,u},Z^{t,x,u})\in\mathbb L^a_{1,2}(0,T;\mathbb R^{n+n+nd})$ by the result of El Karoui et al. \cite[Section 5.2]{kpq97}. 
As a result, 
$$\mathcal T[u]\in\mathbb L^a_{1,2}(0,T;\mathbb R^m).
$$

Next, we prove that the iterative sequence $\{u^k\}_{k\geq0}$ converges in $\mathbb L^a_{1,2}(0,T;\mathbb R^m)$.
We set $u^0\equiv0$ and $u^{k+1}=\mathcal T[u^k]$ when $k\geq0$. 
By Theorem \ref{th:wp}, the sequence $\{u^k\}_{k\geq0}$ is a Cauchy sequence in $\mathscr U_{\mathrm{ad}}[t,T]$. 
We denote $(X^k,Y^k,Z^k):=(X^{t,x,u^k},Y^{t,x,u^k},Z^{t,x,u^k})$ for convenience. 
Therefore, $(X^k,Y^k,Z^k,u^k)\in\mathbb L^a_{1,2}(0,T;\mathbb R^{n+n+nd+m})$ recursively.
Then, for any $i=1,...,d$, a version of $(\mathbf D_{\theta}^iX^k,\mathbf D_{\theta}^iY^k,\mathbf D_{\theta}^iZ^k)$ is given by
\begin{equation}\label{app:lh1}
\left\{
\begin{aligned}
\mathbf D^i_\theta X^k_s = &C_i(\theta) X^k_{\theta}+D_i(\theta) u^k_{\theta}+\sigma_i(\theta)+\int_{\theta}^s\big[A(r)\mathbf D^i_\theta X^k_r+B(r)\mathbf D^i_\theta u^k_r\big]\, \mathrm dr\\
&+\sum_{j=1}^d\int_{\theta}^s \big[C_j(r)\mathbf D^i_\theta X^k_r+D_j(r)\mathbf D^i_\theta u^k_r\big]\, \mathrm dW^j_r,\\
\mathbf D^i_\theta Y^k_s = &\mathrm D^2_{xx}g(X^k_T)\mathbf D^i_\theta X^k_T+\int_s^T\Big[A(r)^\top \mathbf D^i_\theta Y^k_r +C(r)^\top \mathbf D^i_\theta Z^k_r\\
&+\mathrm D^2_{xx}l(r,X_r^k,u^k_r)\mathbf D^i_\theta X^k_r+\mathrm D^2_{xu}l(r,X_r^k,u^k_r)\mathbf D^i_\theta u_r^k\Big]\, \mathrm dr\\
&-\sum_{j=1}^{d}\int^T_s\mathbf D^i_\theta Z^{k,j}_r\, \mathrm dW^j_r,
\end{aligned}
\right.
\end{equation} 
when $s\in[\theta,T]$ and $(\mathbf D^i_\theta X^k_s,\mathbf D^i_\theta Y^k_s,\mathbf D^i_\theta Z^k_s,\mathbf D^i_\theta u^k_s) =0$ when $s\in[0,\theta)$. By the chain rule, we have
$$
\begin{aligned}
\mathbf D^i_\theta u^{k+1}_s=&\mathbf D^i_\theta u^{k}_s-\eta\Big[B(s)^\top\mathbf D^i_\theta Y^{k}_s+D(s)^\top\mathbf D^i_\theta Z^{k}_s\\
&+\mathrm D^2_{uu}l(s,X^k_s,u^k_s)\mathbf D^i_\theta u^{k}_s+\mathrm D^2_{ux}l(s,X^k_s,u^k_s)\mathbf D^i_\theta X^{k}_s\Big].
\end{aligned}
$$

For any $i=1,...,d$, we further introduce the following Hamiltonian system, which admits a unique solution $(X^{\theta,i},Y^{\theta,i},Z^{\theta,i},u^{\theta,i})$ by Theorem \ref{th:wp} and Lemma \ref{lem:2oc}. 
(Recall that $\mathcal R^{t,x},\mathcal Q^{t,x},\mathcal S^{t,x}$ are defined in \eqref{3:not2})
\begin{equation}\label{app:lh2}
\left\{
\begin{aligned}
&\mathrm dX^{\theta,i}_s = \big[A(s)X^{\theta,i}_s+B(s)u^{\theta,i}_s\big]\, \mathrm ds+\sum_{j=1}^d \big[C_j(s)X^{\theta,i}_s+D_j(s)u^{\theta,i}_s\big]\, \mathrm dW^j_s,\\
&\mathrm d Y^{\theta,i}_s=-\bigg[A(s)^\top Y^{\theta,i}_s +C(s)^\top Z^{\theta,i}_s+\mathcal Q_s^{t,x}X^{\theta,i}_s+(\mathcal S_s^{t,x})^\top u^{\theta,i}_s\bigg]\, \mathrm ds+\sum_{j=1}^{d}Z^{\theta,i,j}_s\, \mathrm dW^j_s,\\
& \mathcal R_s^{t,x}u^{\theta,i}_s+\mathcal S_s^{t,x}X^{\theta,i}_s+B(s)^\top Y^{\theta,i}_s+D(s)^\top Z^{\theta,i}_s=0,~~~~s\in[\theta,T],\\
& X^{\theta,i}_{\theta}= C_i(\theta) \bar X^{t,x}_{\theta}+D_i(\theta) \bar u^{t,x}_{\theta}+\sigma_i(\theta),~~~~~~Y^{\theta,i}_T=\mathcal G^{t,x} X^{\theta,i}_T.
\end{aligned}
\right.
\end{equation}
For convenience, we denote $X^{\theta}=(X^{\theta,1},...,X^{\theta,d})$ and similarly for $Y^{\theta},Z^{\theta},u^{\theta}$. 
We would like to prove that $(\mathbf D_\theta X^k,\mathbf D_\theta Y^k,\mathbf D_\theta Z^k,\mathbf D_\theta u^k)$ converges to $(X^{\theta},Y^{\theta},Z^{\theta},u^{\theta})$ when $k\to \infty$.

Applying the estimate \eqref{th:wp:1}, we have
\begin{equation}
\mathbb E\bigg[\sup_{\theta\leq s \leq T}(\vert X^{\theta}_s\vert^2+\vert Y^{\theta}_s\vert^2)+\int_{\theta}^T(\vert Z^{\theta}_s\vert^2+\vert u^{\theta}_s\vert^2)\, \mathrm ds\bigg]\leq K\Big\{1+\mathbb E[\vert \bar X^{t,x}_{\theta}\vert^2+\vert \bar u^{t,x}_{\theta}\vert^2]\Big\}.
\end{equation}
Therefore, \eqref{prop:est:1} implies
\begin{equation}\label{app:est1}
\int_0^T\mathbb E\bigg[\sup_{\theta\leq s \leq T}(\vert X^{\theta}_s\vert^2+\vert Y^{\theta}_s\vert^2)+\int_{\theta}^T(\vert Z^{\theta}_s\vert^2+\vert u^{\theta}_s\vert^2)\, \mathrm ds\bigg]\, \mathrm d\theta\leq K(1+\vert x\vert^2).
\end{equation}

We also introduce the following decoupled FBSDE, which depends on $u^{\theta,i}$ and $(X^k,u^k)$:
\begin{equation}\label{a.5}
\left\{
\begin{aligned}
\mathrm dX^{k,\theta,i}_s =& \big[A(s)X^{k,\theta,i}_s+B(s)u^{\theta,i}_s\big]\, \mathrm ds+\sum_{j=1}^d \big[C_j(s)X^{k,\theta,i}_s+D_j(s)u^{\theta,i}_s\big]\, \mathrm dW^j_s,\\
\mathrm dY^{k,\theta,i}_s =& -\bigg[A(s)^\top Y^{k,\theta,i}_s +C(s)^\top Z^{k,\theta,i}_s+\mathrm D^2_{xx}l(s,X_s^k,u^k_s)X^{k,\theta,i}_s\\
&+\mathrm D^2_{xu}l(s,X_s^k,u^k_s) u^{\theta,i}_s\bigg]\, \mathrm ds+\sum_{j=1}^{d}Z^{k,\theta,i,j}_s\, \mathrm dW^j_s,~~~~s\in[\theta,T],\\
X^{k,\theta,i}_{\theta}=& C_i(\theta) X^{k}_{\theta}+D_i(\theta) u^{k}_{\theta}+\sigma_i(\theta),~~~~~~Y^{k,\theta,i}_T=\mathrm D^2_{xx}g(X^k_T) X^{k,\theta,i}_T.
\end{aligned}
\right.
\end{equation}
We also denote $X^{k,\theta}=(X^{k,\theta,1},...,X^{k,\theta,d})$, similarly for $Y^{k,\theta},Z^{k,\theta}$. 
The above FBSDE incorporates parameters from both the FBSDE \eqref{app:lh1} and Hamiltonian system \eqref{app:lh2}, thereby serving as a transitional bridge between them.

Applying the standard estimates for SDEs and BSDEs, there exists a constant $K>0$ independent of $k$ such that
\begin{equation}\label{app:est2}
  \begin{aligned}
&\mathbb E\bigg[\sup_{\theta\leq s\leq T}(\vert \mathbf D_{\theta}X^k_s-X^{k,\theta}_s\vert^2+\vert \mathbf D_{\theta}Y^k_s-Y^{k,\theta}_s\vert^2)+\int_\theta^T\vert \mathbf D_{\theta}Z^k_s-Z^{k,\theta}_s\vert^2\, \mathrm ds\bigg]\\
\leq& K\mathbb E\int_0^T\vert \mathbf D_{\theta}u^k_s-u^{\theta}_s\vert^2\, \mathrm ds,
\end{aligned}
\end{equation}
and 
\begin{equation}
\mathbb E\bigg[\sup_{\theta\leq s\leq T}(\vert X^{k,\theta}_s-X^{\theta}_s\vert^2+\vert Y^{k,\theta}_s-Y^{\theta}_s\vert^2)+\int_\theta^T\vert Z^{k,\theta}_s-Z^{\theta}_s\vert^2\, \mathrm ds\bigg]\leq K\mathcal I^{k,\theta}_1,
\end{equation}
where 
\begin{align*}
& \mathcal I_1^{k,\theta}=\mathbb E\bigg[\vert X^k_{\theta}-\bar X^{t,x}_{\theta}\vert^2+\vert u^k_{\theta}-\bar u^{t,x}_{\theta}\vert^2+\vert\mathrm D^2_{xx}g(X^k_T)-\mathrm D^2_{xx}g(\bar X^{t,x}_T)\vert^2\vert X^{\theta}_T\vert^2\\
& +\int_{\theta}^{T}\Big[(\vert\mathrm D^2_{xx}l(s,X^k_s,u^k_s)-\mathcal Q^{t,x}_s \vert^2+\vert\mathrm D^2_{ux}l(s,X^k_s,u^k_s)-\mathcal S^{t,x}_s\vert^2) \vert X^{\theta}_s\vert^2\\
&+(\vert\mathrm D^2_{xu}l(s,X^k_s,u^k_s)-(\mathcal S^{t,x}_s)^{\top}\vert^2+\vert\mathrm D^2_{uu}l(s,X^k_s,u^k_s)-\mathcal R^{t,x}_s\vert^2) \vert u^{\theta}_s\vert^2\Big]\, \mathrm d s\bigg].
\end{align*}
By the fact that $(X^k,u^k)$ converges to $(\bar X^{t,x},\bar u^{t,x})$ as $k\to \infty$, \eqref{app:est1} and Lebesgue's dominated convergence theorem,  we have
\begin{equation}\label{app:est4}
\lim_{k\to\infty} \int_0^T\mathcal I^{k,\theta}_1\, \mathrm d\theta =0.
\end{equation}

Moreover, applying Young's inequality and by \eqref{app:est2}, we deduce that, for any $\varepsilon>0$,
\begin{equation}
  \begin{aligned}
&\mathbb E\int_{\theta}^T\vert \mathbf D_{\theta}u^{k+1}_s-u^{\theta}_s\vert^2\, \mathrm ds\\
=&\mathbb E\int_{\theta}^T\vert \mathbf D_{\theta}u^{k}_s-u^{\theta}_s-\eta[B(s)^\top\mathbf D^i_\theta Y^{k}_s+D(s)^\top\mathbf D^i_\theta Z^{k}_s+\mathrm D^2_{uu}l(s,X^k_s,u^k_s)\mathbf D^i_\theta u^{k}_s\\
&+\mathrm D^2_{ux}l(s,X^k_s,u^k_s)\mathbf D^i_\theta X^{k}_s]\vert^2\, \mathrm ds\leq (1+\frac{1}{\varepsilon})\eta^2\mathcal I^{k,\theta}_2+(1+\varepsilon)\mathcal I^{k,\theta}_3,
\end{aligned}
\end{equation}
where
\begin{align*}
\mathcal I^{k,\theta}_2&=\mathbb E\int_{\theta}^{T}\vert B(s)^\top Y^{k,\theta}_s+D(s)^\top Z^{k,\theta}_s+\mathrm D^2_{ux}l(s,X^k_s,u^k_s) X^{k,\theta}_s+\mathrm D^2_{uu}l(s,X^k_s,u^k_s) u^{\theta}_s\vert^2\, \mathrm d s,\\
\mathcal I^{k,\theta}_3&=\mathbb E\int_{\theta}^{T}\vert \mathbf D_{\theta}u^{k}_s-u^{\theta}_s+B(s)^\top (\mathbf D_{\theta}Y^k_s-Y^{k,\theta}_s)+D(s)^\top (\mathbf D_{\theta}Z^k_s-Z^{k,\theta}_s)\\
&+\mathrm D^2_{ux}l(s,X^k_s,u^k_s) (\mathbf D_{\theta}X^k_s-X^{k,\theta}_s)+\mathrm D^2_{uu}l(s,X^k_s,u^k_s) (\mathbf D_{\theta}u^k_s-u^{\theta}_s)\vert^2\, \mathrm d s.
\end{align*}
Noting that $\mathcal R_s^{t,x}u^{\theta}_s+\mathcal S_s^{t,x}X^{\theta}_s+B(s)^\top Y^{\theta}_s+D(s)^\top Z^{\theta}_s=0$, we have
\begin{align*}
  \mathcal I^{k,\theta}_2=&\mathbb E\int_{\theta}^{T}\vert B(s)^\top (Y^{k,\theta}_s-Y^{\theta}_s)+D(s)^\top (Z^{k,\theta}_s-Z^{\theta}_s)+\mathrm D^2_{ux}l(s,X^k_s,u^k_s) (X^{k,\theta}_s-X^{\theta}_s)\\
  &+[\mathrm D^2_{ux}l(s,X^k_s,u^k_s)-\mathcal S^{t,x}_s]X^{\theta}_s+[\mathrm D^2_{uu}l(s,X^k_s,u^k_s)-\mathcal R^{t,x}_s] u^{\theta}_s\vert^2\, \mathrm d s\\
  \leq&K\mathcal I^{k,\theta}_1.
\end{align*}
Similarly to the proof of Theorem \ref{th:wp}, applying It\^o's formula to $\langle Y^{k,\theta}_s-\mathbf D_{\theta}Y^{k}_s,X^{k,\theta}_s-\mathbf D_{\theta}X^{k}_s \rangle$ on $[\theta,T]$ and by \eqref{a.5}, we deduce
\begin{equation}\label{app:est3}
  \mathcal I^{k,\theta}_3\leq (1-2\eta\delta+\eta^2K)\mathbb E\int_{\theta}^T\vert \mathbf D_{\theta}u^{k}_s-u^{\theta}_s\vert^2\, \mathrm ds.
\end{equation}
If necessary, we could choose a bigger $K$ in \eqref{pf:th:wp:2} and \eqref{app:est3} such that $\eta=\frac{\delta}{K}$ is the same one. Then, we obtain
$$
  \mathcal I^{k,\theta}_3\leq (1-\frac{\delta^2}{K})\mathbb E\int_{\theta}^T\vert \mathbf D_{\theta}u^{k}_s-u^{\theta}_s\vert^2\, \mathrm ds.
$$
Taking $\varepsilon$ small enough such that $\gamma=(1-\frac{\delta^2}{K})(1+\varepsilon)<1$, we derive 
\begin{equation}
\mathbb E\int_0^T\int_{0}^T\vert \mathbf D_{\theta}u^{k+1}_s-u^{\theta}_s\vert^2\, \mathrm ds\, \mathrm d\theta\leq   K\int_0^T\mathcal I^{k,\theta}_1\, \mathrm d\theta+\gamma\mathbb E\int_0^T\int_{0}^T\vert \mathbf D_{\theta}u^{k}_s-u^{\theta}_s\vert^2\, \mathrm ds\, \mathrm d\theta,
\end{equation}
where $K$ and $\gamma$ are constants independent of $k$. Combining \eqref{pf:th:wp:2} and \eqref{app:est4}, for any $\varepsilon>0$, there exists a sufficiently large integer $N>0$ such that, for any $k\geq N$,
\begin{align*}
&\mathbb E\int_0^T\vert u^{k+1}_s-\bar u^{t,x}_s\vert^2\, \mathrm ds+\mathbb E\int_0^T\int_{0}^T\vert \mathbf D_{\theta}u^{k+1}_s-u^{\theta}_s\vert^2\, \mathrm ds\, \mathrm d\theta\\
\leq &\varepsilon+\gamma\bigg[\mathbb E\int_0^T\vert u^{k}_s-\bar u^{t,x}_s\vert^2\, \mathrm ds+\mathbb E\int_0^T\int_{0}^T\vert \mathbf D_{\theta}u^{k}_s-u^{\theta}_s\vert^2\, \mathrm ds\, \mathrm d\theta\bigg]\\
\leq &\frac{\varepsilon}{1-\gamma}+\gamma^{k+1-N}\bigg[\mathbb E\int_0^T\vert u^{N}_s-\bar u^{t,x}_s\vert^2\, \mathrm ds+\mathbb E\int_0^T\int_{0}^T\vert \mathbf D_{\theta}u^{N}_s-u^{\theta}_s\vert^2\, \mathrm ds\, \mathrm d\theta\bigg].
\end{align*}
Thus, we deduce that $\{u^{k}\}_{k\geq0}$ is also a Cauchy sequence under the norm $\Vert\cdot\Vert_{\mathbb L^a_{1,2}}$. 
Since $\mathbb L^a_{1,2}(0,T;\mathbb R^m)$ is closed under the norm $\Vert\cdot\Vert_{\mathbb L^a_{1,2}}$, 
we conclude that the limit $\bar u^{t,x}\in\mathbb L^a_{1,2}(0,T;\mathbb R^m)$ and $(\bar X^{t,x},\bar Y^{t,x},\bar Z^{t,x})\in\mathbb L^a_{1,2}(0,T;\mathbb R^{n+n+nd})$ correspondingly (see \eqref{app:est2}).  
Since
$$
(\mathbf D_{\theta}X^{k},\mathbf D_{\theta}Y^{k},\mathbf D_{\theta}Z^{k},\mathbf D_{\theta}u^{k})\to (X^{\theta},Y^{\theta},Z^{\theta},u^{\theta}),
$$
as $k\to\infty$, \eqref{app:lh2} also provides a version for $(\mathbf D_{\theta}^i\bar X^{t,x},\mathbf D_{\theta}^i\bar Y^{t,x},\mathbf D^i_{\theta}\bar Z^{t,x},\mathbf D^i_{\theta}\bar u^{t,x})$. 

Finally, we prove that $\bar Z^{t,x}_s=\mathbf D_s\bar Y^{t,x}_s$.
Note that $\bar Y^{t,x}$ satisfies the following ``forward'' form:
 \begin{align*}
\bar Y^{t,x}_s=&\bar Y^{t,x}_t-\int_{t}^s\big[A(r)^\top \bar Y^{t,x}_r +C(r)^\top \bar Z^{t,x}_r+\mathrm D_xl(r,\bar X^{t,x}_r,\bar u^{t,x}_r)\big]\, \mathrm dr+\int_{t}^s\bar Z^{t,x,j}_r\, \mathrm dW_r^j.
\end{align*}
By El Karoui et al. \cite[Lemma 5.1]{kpq97}, we have, for any $i=1,...,d$,
 \begin{align*}
\mathbf D_{\theta}^i\bar Y^{t,x}_s=&\bar Z^{t,x,i}_{\theta}-\int_{\theta}^s[A(r)^\top \mathbf D^i_\theta \bar Y^{t,x}_r +C(r)^\top \mathbf D^i_\theta \bar Z^{t,x}_r+\mathcal Q_r^{t,x}\mathbf D^i_\theta \bar X^{t,x}_r+(\mathcal S_r^{t,x})^\top \mathbf D^i_\theta \bar u^{t,x}_r]\, \mathrm dr\\
&+\int_{\theta}^s\mathbf D^i_\theta \bar Z^{t,x,j}_r\, \mathrm dW_r^j.
\end{align*}
Taking $s=\theta$ yields $D_{\theta}\bar Y^{t,x}_{\theta}=\bar Z_{\theta}^{t,x}$. The proof is completed.
  \section{Proof of Lemma \ref{lem:dpp}}\label{APP-B} 
We first assume $\xi=\sum_{i=1}^{\infty}x_i\mathbbm 1_{A_i}$, where $x_i\in\mathbb R^n,~i\geq1$ are constants and $\{A_i\}_{i\geq 1}\subset\mathcal F_t$ is a partition of $\Omega$. 
Then, by Remark \ref{rmk} and $(X^{t,\xi},u^{t,\xi})=\sum_{i=1}^{\infty}(X^{t,x_i},u^{t,x_i})\mathbbm 1_{A_i}$, we obtain
\begin{align*}
    V(t,\xi)&=V(t,\sum_{i=1}^{\infty}x_i\mathbbm 1_{A_i})=\sum_{i=1}^{\infty}V(t,x_i)\mathbbm 1_{A_i}\\
    &=\sum_{i=1}^{\infty}\mathbb E\left[ g(\bar X^{t,x_i}_{T})+\int_t^{T}l(s,\bar X^{t,x_i}_s,\bar u^{t,x_i}_s)\, \mathrm d s\right]\mathbbm 1_{A_i}\\
    &=\sum_{i=1}^{\infty}\mathbb E\left[ g(\bar X^{t,x_i}_{T})+\int_t^{T}l(s,\bar X^{t,x_i}_s,\bar u^{t,x_i}_s)\, \mathrm d s\bigg\vert\mathcal F_t\right]\mathbbm 1_{A_i}\\
    &=\mathbb E\left[ \sum_{i=1}^{\infty}(g(\bar X^{t,x_i}_{T})\mathbbm 1_{A_i})+\int_t^{T}\sum_{i=1}^{\infty}(l(s,\bar X^{t,x_i}_s,\bar u^{t,x_i}_s)\mathbbm 1_{A_i})\, \mathrm d s\bigg\vert\mathcal F_t\right]\\
    &=\mathbb E\left[ g(\bar X^{t,\xi}_{T})+\int_t^{T}l(s,\bar X^{t,\xi}_s,\bar u^{t,\xi}_s)\, \mathrm d s\bigg\vert\mathcal F_t\right].
\end{align*}

In the general case where $\xi\in L^2_{\mathcal F_t}(\Omega;\mathbb R^n)$, we can find a sequence $\{\xi_k\}_{k\geq1}$ such that each $\xi_k$ takes countably many values and $\vert\xi_k-\xi\vert\leq\frac1k$. 
Based on the above analysis, we have for any $k\geq1$,
$$
V(t,\xi_k)=\mathbb E\left[ g(\bar X^{t,\xi_k}_{T})+\int_t^{T}l(s,\bar X^{t,\xi_k}_s,\bar u^{t,\xi_k}_s)\, \mathrm d s\bigg\vert\mathcal F_t\right].
$$
By Proposition \ref{prop:vfd}, we have $\mathrm D_xV$ is uniformly Lipschitz continuous in $x$. Then,
\begin{equation}
\begin{aligned}
   \mathbb E[\vert V(t,\xi_k)-V(t,\xi)\vert]&=\mathbb E\left[\left\vert\int_0^1\langle\mathrm D_xV(t,\xi+\gamma(\xi_k-\xi)),\xi_k-\xi\rangle\, \mathrm d\gamma\right\vert\right]\\
   &\leq \frac{K}{k}(1+\frac{1}{k}+\mathbb E[\vert\xi\vert])\to0,~~\mbox{ as } k\to\infty.
\end{aligned}
\end{equation}
Moreover,
\begin{align*}
  &\mathbb E\bigg[\bigg\vert V(t,\xi_k)-\mathbb E\left[ g(\bar X^{t,\xi}_{T})+\int_t^{T}l(s,\bar X^{t,\xi}_s,\bar u^{t,\xi}_s)\, \mathrm d s\bigg\vert\mathcal F_t\right]\bigg\vert\bigg]\\
  \leq&\mathbb E\left[\vert g(\bar X^{t,\xi_k}_{T})-g(\bar X^{t,\xi}_{T})\vert+\int_t^T\vert l(s,\bar X^{t,\xi_k}_s,\bar u^{t,\xi_k}_s)-l(s,\bar X^{t,\xi}_s,\bar u^{t,\xi}_s)\vert\, \mathrm d s\right]\\
  \leq&K\mathbb E\int_t^T\big[(1+\vert \bar X^{t,\xi}_{s}\vert+\vert\bar X^{t,\xi_k}_{s}\vert+\vert \bar u^{t,\xi}_{s}\vert+\vert\bar u^{t,\xi_k}_{s}\vert)(\vert\bar X^{t,\xi_k}_{s}-\bar X^{t,\xi}_{s}\vert+\vert\bar u^{t,\xi_k}_{s}-\bar u^{t,\xi}_{s}\vert)\big]\, \mathrm d s\\
  &+K \mathbb E\left[(1+\vert \bar X^{t,\xi}_{T}\vert+\vert\bar X^{t,\xi_k}_{T}\vert)\vert\bar X^{t,\xi_k}_{T}-\bar X^{t,\xi}_{T}\vert\right]\\
  \leq &K\left(1+\mathbb E\bigg[\sup_{t\leq s \leq T}\vert X^{t,\xi}_s\vert^2+\sup_{t\leq s \leq T}\vert X^{t,\xi_k}_s\vert^2\bigg]+\mathbb E\int_t^T\big[\vert u^{t,\xi}_s\vert^2\, +\vert u^{t,\xi_k}_s\vert^2\big]\, \mathrm ds\right)^{1/2}\\
  &\times\left(\mathbb E\bigg[\sup_{t\leq s \leq T}\vert X^{t,\xi_k}_s-X^{t,\xi}_s\vert^2\bigg]+\mathbb E\int_t^T\vert u^{t,\xi_k}_s-u^{t,\xi}_s\vert^2\, \mathrm ds\right)^{1/2}.
\end{align*}
By the estimate \eqref{2:est}, we have
$$
\mathbb E\left[\sup_{t\leq s \leq T}\vert \bar X^{t,\xi}_s\vert^2\right]+\mathbb E\int_t^T\vert \bar u^{t,\xi}_s\vert^2\, \mathrm ds\leq K(1+\mathbb E[\vert\xi\vert^2])
$$
and 
$$
\mathbb E\left[\sup_{t\leq s \leq T}\vert \bar X^{t,\xi_k}_s\vert^2\right]+\mathbb E\int_t^T\vert \bar u^{t,\xi_k}_s\vert^2\, \mathrm ds\leq K(1+\mathbb E[\vert\xi_k\vert^2])\leq K(1+\frac{1}{k^2}+\mathbb E[\vert\xi\vert^2]).
$$
Similarly to the proof of the estimate \eqref{prop:est:2} and by the estimate \eqref{2:est} again, we derive
\begin{equation}
  \mathbb E\left[\sup_{t\leq s \leq T}\vert \bar X^{t,\xi_k}_s-\bar X^{t,\xi}_s\vert^2\right]+\mathbb E\int_t^T\vert \bar u^{t,\xi_k}_s-\bar u^{t,\xi}_s\vert^2\, \mathrm ds\leq K\mathbb E[\vert \xi_k-\xi\vert^2]\leq \frac{K}{k^2}.
\end{equation}
Therefore, 
$$
\mathbb E\bigg[\bigg\vert V(t,\xi_k)-\mathbb E\left[ g(\bar X^{t,\xi}_{T})+\int_t^{T}l(s,\bar X^{t,\xi}_s,\bar u^{t,\xi}_s)\, \mathrm d s\bigg\vert\mathcal F_t\right]\bigg\vert\bigg]\to0, \mbox{ when } k\to\infty.
$$
By the uniqueness of the limit, we conclude the desired result.
\section{Convexity of value function}\label{APP-C}
For the convexity of the value function $V$, we provide the following sufficient condition.
\begin{lemma}\label{lem:covf}
  Suppose that Assumptions \ref{A3} and \ref{A4} hold. 
  Further assume that $g,l$ are convex in $(x,u)$ for almost all $(t,\omega)\in[0,T]\times\Omega$.
  Then, the value function $V$ is convex in $x$ for any $t\in[0,T]$.
\end{lemma}
\begin{proof}
For any $x_0,x_1\in\mathbb R^n$ and any $\lambda\in(0,1)$, we denote $x_{\lambda}=\lambda x_1+(1-\lambda)x_0$.
By the linearity of the controlled SDE, we obtain
$$
X^{t,x_{\lambda},\lambda \bar u^{t,x_1}+(1-\lambda)\bar u^{t,x_0}}=\lambda\bar X^{t,x_1}+(1-\lambda)\bar X^{t,x_0}.
$$
By the definition of $V$ and the convexity of $g$ and $l$, we derive that
\begin{align*}
V(t,x_{\lambda})\leq &J(t,x_{\lambda};\lambda \bar u^{t,x_1}+(1-\lambda)\bar u^{t,x_0})\\
=&\mathbb E\Bigg[g(\lambda \bar X^{t,x_1}_T+(1-\lambda)\bar X^{t,x_0}_T)\\
&+\int_t^Tl(s,\lambda \bar X^{t,x_1}_s+(1-\lambda)\bar X^{t,x_0}_s,\lambda \bar u^{t,x_1}_s+(1-\lambda)\bar u^{t,x_0}_s)\mathrm{d}s\Bigg]\\
\leq&\lambda \mathbb E\Bigg[g(\bar X^{t,x_1}_T)+\int_t^Tl(s,\bar X^{t,x_1}_s,\bar u^{t,x_1}_s)\mathrm{d}s\Bigg]\\
&+(1-\lambda)\mathbb E\Bigg[ g(\bar X^{t,x_0}_T)+\int_t^Tl(s,\bar X^{t,x_0}_s,\bar u^{t,x_0}_s)\mathrm{d}s\Bigg]\\
=&\lambda J(t,x_1;\bar u^{t,x_1})+(1-\lambda)J(t,x_0;\bar u^{t,x_0})\\
=&\lambda V(t,x_1)+(1-\lambda)V(t,x_0).
\end{align*}
The proof is completed.
\end{proof}

\end{appendix}

\bigskip

\noindent{\bf Acknowledgments.} 
The authors are very grateful to Prof. Jiongmin Yong, University of Central Florida, for his profound insights into the quasi-Riccati equations and multiple valuable discussions.
The authors also thank Prof. Xun Li, The Hong Kong Polytechnic University, and Dr. Hanxiao Wang, Shenzhen University, for many helpful discussions and suggestions. 
This work was supported in part by the Natural Science Foundation of Shandong Province (ZR2024ZD35 and ZR2025MS63) and the National Natural Science Foundation of China (12271304).

\end{document}